\documentclass{amsart}

\usepackage[all]{xy}

\usepackage{latexsym}

\usepackage{amssymb}
\usepackage{amsfonts}
\usepackage{amscd}
\usepackage{amsmath,amsthm}

\newtheorem{lemma}{Lemma}[section]

\newtheorem{theorem}[lemma]{Theorem}
\newtheorem{corollary}[lemma]{Corollary}

\newtheorem*{introthm}{Theorem}

\newtheorem{prop}[lemma]{Proposition}
\newtheorem{thm}[lemma]{Theorem}
\newtheorem{cor}[lemma]{Corollary}

{

}

\theoremstyle{definition}

\theoremstyle{remark}

\newcommand{\Hom}{\operatorname{Hom}}

\numberwithin{equation}{section}

\begin{document}

\title{The geometry of arithmetic noncommutative projective lines}
\author{A. Nyman}
\address{Department of Mathematics, 516 High St, Western Washington University, Bellingham, WA 98225-9063}
\email{adam.nyman@wwu.edu}
\date{\today}

\begin{abstract}
Let $k$ be a perfect field and let $K/k$ be a finite extension of fields.  An arithmetic
noncommutative projective line is a noncommutative space of the
form ${\sf Proj }\mathbb{S}_{K}(V)$, where $V$ be a $k$-central two-sided vector space over $K$ of rank two and $\mathbb{S}_{K}(V)$ is the noncommutative symmetric
algebra generated by $V$ over $K$ defined by M. Van den Bergh \cite{vandenbergh}.   We study the geometry of these spaces.  More precisely, we prove they are integral, we classify vector bundles over them, we classify them up to isomorphism, and we classify isomorphisms between them.  Using the classification of isomorphisms, we compute the automorphism group of an arithmetic noncommutative projective line.
\end{abstract}

\maketitle

\tableofcontents

\pagenumbering{arabic}

\section{Introduction}
Throughout this paper, $k$ will denote a perfect field and $K/k$ will be a finite extension of fields.  The purpose of this paper is to study the geometry of noncommutative spaces (i.e. Grothendieck categories) of the form ${\sf Proj }\mathbb{S}_{K}(V)$, where $V$ is a $k$-central two-sided vector space over $K$ of rank two, i.e. a $k$-central $K-K$-bimodule which is two-dimensional as a vector space via the right and left actions of $K$ on $V$, $\mathbb{S}_{K}(V)$ is the noncommutative symmetric algebra defined by Van den Bergh \cite{vandenbergh}, and ${\sf Proj }\mathbb{S}_{K}(V)$ denotes the quotient of the category of graded right $\mathbb{S}_{K}(V)$-modules modulo the full subcategory of direct limits of right bounded modules.  We denote the noncommutative space by $\mathbb{P}_{K}(V)$, and we refer to this space as an {\it arithmetic} noncommutative projective line since, as we shall show, its geometry is intimately connected to data associated with $K/k$.  In the sequel, we shall drop the term ``arithmetic".  We hope that noncommutative projective lines have some utility as basic examples arising in the relatively new field of noncommutative arithmetic geometry.

Our primary motivation for the study of noncommutative projective lines is Artin's conjecture \cite{artin}, which states that the division ring of fractions of a noncommutative surface not finite over its center is the function field of a noncommutative $\mathbb{P}^{1}$-bundle over a smooth commutative curve (see \cite{vandenbergh} for a definition of the latter space).  The investigations in this paper do not directly address this conjecture, but provide what physicists might call a toy model for the geometry of noncommutative $\mathbb{P}^{1}$-bundles over smooth commutative curves, since a noncommutative projective line is just a noncommutative $\mathbb{P}^{1}$-bundle over the point $\mbox{Spec }K$.  In Theorem \ref{thm.classify}, stated below, we obtain isomorphism invariants of noncommutative projective lines.  It is not yet known whether these are also birational invariants.  If Artin's conjecture is true, then noncommutative surfaces infinite over their center will share birational invariants with noncommutative $\mathbb{P}^{1}$-bundles over smooth commutative curves, and we hope our work suggests the form these invariants take.

In order to justify the name `noncommutative projective line', these spaces should have geometric properties in common with the commutative projective line, and they do: it is known that they are noetherian \cite[Section 6.3]{vandenbergh} $\mbox{Ext }$-finite \cite[Corollary 3.6]{finite} (homological) dimension one categories \cite[Theorem 15.4]{chan} having a Serre functor \cite[Theorem 16.4]{chan}.  In this paper, we prove the following

\begin{introthm}
Noncommutative projective lines are integral in the sense of \cite{Smith} (Corollary \ref{cor.integral}).  The line bundles over a noncommutative projective line are indexed by $\mathbb{Z}$ and every vector bundle over a noncommutative projective line is a direct sum of line bundles (Corollary \ref{cor.grothendieck}).
\end{introthm}

Using our classification of vector bundles, we classify noncommutative projective lines up to $k$-linear equivalence.  To describe the classification, we introduce some notation.  We let $V^{*}$ denote the right dual of a two-sided vector space $V$ (see Section \ref{subsection.twosidedvectorspaces}).  If $\sigma \in \operatorname{Gal}(K/k)$, we let $K_{\sigma}$ denote the two-sided vector space whose left $K$ action is ordinary multiplication in $K$ and whose right $K$ action is defined as $x \cdot a := x\sigma(a)$.  We prove the following (Theorem \ref{thm.twosidedisom}):

\begin{introthm}
There is a $k$-linear equivalence $\mathbb{P}_{K}(V) \rightarrow \mathbb{P}_{K}(W)$ if and only if there exists $\delta, \epsilon \in \operatorname{Gal}(K/k)$ such that either
$$
V \cong K_{\delta} \otimes_{K} W \otimes_{K} K_{\epsilon}
$$
or
$$
V \cong K_{\delta} \otimes_{K} W^{*} \otimes_{K} K_{\epsilon}.
$$
\end{introthm}
The ``if" direction was proven, in a much more general context, in \cite{izuru}.  This result can be viewed as a noncommutative analogue of \cite[Proposition 2.2, p. 370]{harts}.

In \cite{papp}, simple two-sided vector spaces are classified.  Since this classification will be invoked in what follows, we recall it now.  Let $\overline{K}$ be an algebraic closure of $K$, let $\operatorname{Emb}(K)$ denote the set of $k$-algebra maps $K \rightarrow \overline{K}$, and let $G=\operatorname{Gal}(\overline{K}/K)$.  Now, $G$ acts on $\operatorname{Emb}(K)$ by left composition. Given $\lambda\in \operatorname{Emb}(K)$, we denote the orbit of $\lambda$ under this action by $\lambda^G$.  We denote the set of finite orbits of $\operatorname{Emb}(K)$ under the action
of $G$ by $\Lambda(K)$.

\begin{thm} \label{thm.papp} \cite{papp} There is a one-to-one correspondence
between $\Lambda(K)$ and isomorphism classes of simple left (and hence right) finite dimensional
two-sided vector spaces.
\end{thm}

If $\lambda \in \operatorname{Emb}(K)$, we may construct from it a canonical two-sided vector space, $V(\lambda)$, as follows: we let $K(\lambda)$ denote the composite of $K$ and $\lambda(K)$.  We let the underlying set of $V(\lambda)$ be $K(\lambda)$ and we define the $K \otimes_{k} K$-structure to be that induced by the formula $a \cdot v \cdot b := av\lambda(b)$.
By \cite[Proposition 2.3]{hart}, the correspondence in Theorem \ref{thm.papp} sends $\lambda^{G}$ to the class of $V(\lambda)$.  The canonical form $V(\lambda)$ will be employed often in the sequel.

In \cite[Theorem 1.3]{P}, David Patrick describes the form of two-sided vector spaces of rank 2.  Using Patrick's result and Theorem \ref{thm.papp} allows us to sharpen our classification of noncommutative projective lines as follows (Theorem \ref{thm.arithclass}):

\begin{thm} \label{thm.classify}
Suppose $\operatorname{char }k \neq 2$ and let $V_{i}$ be a rank 2 two-sided vector space for $i=1,2$.  There is a $k$-linear equivalence $\mathbb{P}_{K}(V_{1}) \rightarrow \mathbb{P}_{K}(V_{2})$ if and only if
\begin{enumerate}
\item{} there exists $\sigma_{i} \in
\operatorname{Gal}(K/k)$ such that $V_{i} \cong K_{\sigma_{i}} \oplus K_{\sigma_{i}}$.  In this case, $\mathbb{P}(V_{i})$ is equivalent to the commutative projective line over $K$.

\item{} There exists $\sigma_{i}, \tau_{i}
\in \operatorname{Gal}(K/k)$, with $\sigma_{i} \neq \tau_{i}$, $V_{i} \cong K_{\sigma_{i}} \oplus K_{\tau_{i}}$ and under the (right) action of ${Gal(K/k)}^{2}$ on itself defined by $(\sigma, \tau) \cdot (\delta, \epsilon):= (\delta^{-1} \sigma \epsilon,\delta^{-1} \tau \epsilon)$ the orbit of $(\sigma_{1},\tau_{1})$ contains an element of the set
$$
\{(\sigma_{2},\tau_{2}), (\sigma_{2}^{-1},\tau_{2}^{-1}), (\tau_{2},\sigma_{2}), (\tau_{2}^{-1}, \sigma_{2}^{-1})\}.
$$

\item{}  $V_{i} \cong V(\lambda_{i})$, and under the action of ${Gal(K/k)}^{2}$ on $\Lambda(K)$ defined by $\lambda^{G} \cdot (\delta, \epsilon):= (\overline{\delta}^{-1} \lambda \epsilon)^{G}$, the orbit of $\lambda_{1}^{G}$ contains either $\lambda_{2}^{G}$ or $\mu_{2}^{G}$ (where $\overline{\delta}$ denotes an extension of $\delta$ to $\overline{K}$, $\overline{\lambda_{2}}$ denotes an extension of $\lambda_{2}$ to $\overline{K}$ and $\mu_{2}:=(\overline{\lambda_{2}})^{-1}|_{K}$).
\end{enumerate}
\end{thm}

Next, we turn our attention to classifying $k$-linear equivalences $\mathbb{P}_{K}(V) \rightarrow \mathbb{P}_{K}(W)$ up to isomorphism.  To describe our main result in this direction, we must recall the definition of three types of canonical equivalences between noncommutative projective lines, studied in Section \ref{section.canonical}.  If $\delta, \epsilon \in \operatorname{Gal}(K/k)$ and we define
$$
\zeta_{i} = \begin{cases} \delta \mbox{ if $i$ is even} \\ \epsilon \mbox{ if $i$ is odd,} \end{cases}
$$
then twisting by the sequence $\{K_{\zeta_{i}}\}_{i \in \mathbb{Z}}$ in the sense of \cite[Section 3.2]{vandenbergh} induces an equivalence of categories $T_{\delta,\epsilon}:\mathbb{P}_{K}(V) \rightarrow \mathbb{P}_{K}(K_{\delta^{-1}} \otimes_{K} V \otimes_{K} K_{\epsilon})$ (see Section \ref{section.t} for more details).  Next, let $\phi:V \rightarrow W$ denote an isomorphism of two-sided vector spaces.  Then $\phi$ induces an isomorphism $\mathbb{S}_{K}(V) \rightarrow \mathbb{S}_{K}(W)$, which in turn induces an equivalence $\Phi:\mathbb{P}_{K}(V) \rightarrow \mathbb{P}_{K}(W)$.  Finally, shifting graded modules by an integer, $i$, induces an equivalence $[i]:\mathbb{P}_{K}(V) \rightarrow \mathbb{P}_{K}(V^{i*})$, and we have the following classification of equivalences between noncommutative projective lines (Theorem \ref{thm.finalmain}):

\begin{introthm}
Suppose $\operatorname{char }k \neq 2$ and let $F:\mathbb{P}_{K}(V) \rightarrow \mathbb{P}_{K}(W)$ be a $k$-linear equivalence.  Then there exists $\delta, \epsilon \in \mbox{Gal }(K/k)$, an isomorphism $\phi:K_{\delta}^{-1} \otimes_{K} V \otimes_{K} K_{\epsilon} \rightarrow W^{i^*}$ inducing an equivalence $\Phi: \mathbb{P}_{K}(K_{\delta^{-1}} \otimes_{K} V \otimes_{K} K_{\epsilon}) \rightarrow \mathbb{P}_{K}(W^{i*})$ and an integer $i$ such that
$$
F \cong [-i] \circ \Phi  \circ T_{\delta, \epsilon}.
$$
Furthermore, $\delta$, $\epsilon$ and $i$ are unique up to natural equivalence, while $\Phi$ is naturally equivalent to $\Phi'$ if and only if there exist nonzero $a, b \in K$ such that $\phi' \phi^{-1}(w) = a \cdot w \cdot b$ for all $w \in W^{i*}$.
\end{introthm}
We conclude the paper with a computation of the automorphism group of a noncommutative projective line. We define the {\it automorphism group of }$\mathbb{P}_{K}(V)$, denoted $\operatorname{Aut }\mathbb{P}_{K}(V)$, to be the set of $k$-linear shift-free equivalences $\mathbb{P}_{K}(V) \rightarrow \mathbb{P}_{K}(V)$ modulo natural equivalence, with composition induced by composition of functors.  By specializing the classification of equivalences to the case that $V=W$, we compute $\operatorname{Aut }\mathbb{P}_{K}(V)$ (Lemma \ref{lemma.aut}, Lemma \ref{lemma.stab}, and Theorem \ref{thm.finalfinalfinal}).  To describe our computation, we first introduce some notation.  Let $\operatorname{Stab}V$ denote the subgroup of $\mbox{Gal }(K/k)^{2}$ consisting of $(\delta, \epsilon)$ such that $K_{\delta^{-1}} \otimes_{K} V \otimes_{K} K_{\epsilon} \cong V$ and let $\operatorname{Aut }V$ denote the set of two-sided vector space isomorphisms $V \rightarrow V$ modulo the relation defined by setting $\phi' \equiv \phi$ if and only if there exist nonzero $a, b \in K$ such that $\phi'  \phi^{-1}(v)=a \cdot v \cdot b$ for all $v \in V$.

\begin{introthm}
Suppose $\operatorname{char }k \neq 2$.  There exists a group homomorphism $\theta:\operatorname{Stab}V^{op} \rightarrow \operatorname{Aut }(\operatorname{Aut }(V))$ (described before Theorem \ref{thm.finalfinalfinal}), such that
$$
\operatorname{Aut }\mathbb{P}_{K}(V) \cong \operatorname{Aut }V \rtimes_{\theta} \operatorname{Stab}V^{op}.
$$
Furthermore, the factors of $\operatorname{Aut}\mathbb{P}_{K}(V)$ are described as follows (where $\sigma, \tau \in \operatorname{Gal}(K/k)$ and $\sigma \neq \tau$):
\begin{enumerate}

\item{} If $V \cong K_{\sigma} \oplus K_{\sigma}$ then $\operatorname{Aut}V \cong \mbox{PGL}_{2}(K)$ and $\operatorname{Stab}V \cong \mbox{Gal }(K/k)$.

\item{} If $V \cong K_{\sigma} \oplus K_{\tau}$ then
$$
\operatorname{Aut}V \cong K^{*} \times K^{*}/\{(a\sigma(b),a\tau(b))|a, b \in K^{*}\}.
$$
and
$$
\operatorname{Stab}V = \{(\delta,\epsilon)|\{\delta^{-1} \sigma \epsilon, \delta^{-1} \tau \epsilon \}=\{\sigma, \tau\}\},
$$
\item{} If $V \cong V(\lambda)$, and if, for $\delta \in \operatorname{Gal }(K/k)$, $\overline{\delta}$ denotes an extension of $\delta$ to $\overline{K}$, then
$$
\operatorname{Aut }(V) \cong K(\lambda)^{*}/K^{*}\lambda(K)^{*}.
$$
and
$$
\operatorname{Stab}(V) = \{(\delta, \epsilon) \in \operatorname{Gal}(K/k)^{2} | (\overline{\delta}^{-1}\lambda \epsilon)^{G}=\lambda^{G}\},
$$
\end{enumerate}
\end{introthm}

There are many notions of noncommutative projective line in the literature and we discuss some of them briefly.  The subject of noncommutative curves of genus zero is studied comprehensively by D. Kussin in \cite{kussin}.  In \cite{whengenus}, we use the results of this paper to prove that noncommutative projective lines are indeed noncommutative curves of genus zero, and we find necessary and sufficient conditions for noncommutative curves of genus zero to be noncommutative projective lines.  We then use results from this paper to address some problems posed in \cite{kussin}.  Another family of noncommutative projective lines was introduced and studied by D. Piontkovski \cite{pion}.  The noncommutative members of this family are non-noetherian noncommutative spaces, hence, are distinct from the spaces we study.  S. P. Smith and J. J. Zhang study another family of noncommutative spaces, $\{\mathbb{V}_{n}^{1}\}_{n \in \mathbb{N}}$, analogous to the commutative projective line \cite{smithzhang}.  However, unless $n=2$, their curves do not admit a Serre functor \cite[Section 7.3]{staffordvdb}.  Noncommutative curves (not necessarily of genus zero) are classified by I. Reiten and M. Van den Bergh in \cite{rvdb} over an algebraically closed field.  Hence, except for the case of a commutative projective line, the spaces we study do not fit into their classification.

{\it Notation and conventions:}  Throughout, $V$ and $W$ will denote two-sided vector spaces of finite rank (see Section \ref{subsection.twosidedvectorspaces} for precise definitions of these terms).  Starting in Section \ref{section.groth}, we will assume that $V$ and $W$ have rank $2$.  Unadorned tensor products will be over $K$.  If $\sigma \in \operatorname{Gal}(K/k)$ or $\lambda \in \operatorname{Emb}(K)$, we let $\overline{\sigma}$ and $\overline{\lambda}$ denote extensions to $\overline{K}$.

If $A$ is a $\mathbb{Z}$-algebra (see \cite[Section 2]{quadrics} for a definition) and $i \in \mathbb{Z}$, we let $e_{i}A$ denote the graded right $A$-module $\oplus_{j \in \mathbb{Z}}A_{ij}$ and we let $Ae_{i}$ denote the graded left $A$-module $\oplus_{j \in \mathbb{Z}}A_{ji}$.  If $A$ is either a right noetherian $\mathbb{Z}$-graded algebra or a right noetherian $\mathbb{Z}$-algebra, we let ${\sf Gr }A$ denote the category of graded right $A$-modules, and we let ${\sf gr }A$ denote the full subcategory of ${\sf Gr }A$ consisting of noetherian objects.  Following \cite[Section 2]{az} and \cite[Section 2]{quadrics}, we let ${\sf Tors }A$ denote the full subcategory of ${\sf Gr }A$ consisting of objects which are direct limits of right bounded modules, and we let ${\sf Proj }A$ denote the quotient of ${\sf Gr }A$ by ${\sf Tors }A$.  Furthermore, we let $\pi_{A}:{\sf Gr }A \rightarrow {\sf Proj }A$ denote the quotient functor.  We sometimes write $\pi$ instead of $\pi_{A}$ where no confusion arises.  Similarly, we let $\tau_{A}:{\sf Gr }A \rightarrow {\sf Tors }A$, or just $\tau$, denote the torsion functor, and we let $\omega_{A}:{\sf Proj }A \rightarrow {\sf Gr }A$, or just $\omega$ denote the section functor.  Unless stated otherwise, the term {\it $A$-module} will mean {\it graded right $A$-module.}

For the remainder of the paper, we will write $\mathbb{P}(V)$ (instead of $\mathbb{P}_{K}(V)$) for a noncommutative projective line defined by $V$.  Similarly, we write $\mathbb{S}(V)$ (instead of $\mathbb{S}_{K}(V)$) for the noncommutative symmetric algebra generated by $V$.

{\it Acknowledgement:}  I thank S. P. Smith for showing me, in 1998, a ``noncommutative" proof of the classification of coherent sheaves over the {\it commutative} projective line.  His proof technique works in our context and is used to prove Theorem \ref{thm.grothendieck}.

\section{Two-sided vector spaces and noncommutative symmetric algebras}

\subsection{Two-sided vector spaces} \label{subsection.twosidedvectorspaces}
By a \emph{two-sided vector
space} we mean a $K \otimes_{k}K$-module $V$.  By {\it right (resp. left) multiplication by $K$} we mean multiplication by elements in $1 \otimes_{k} K$ (resp. $K \otimes_{k} 1$).  We denote the restriction of scalars of $V$ to $K \otimes_{k} 1$ (resp. $1\otimes_{k} K$) by ${}_{K}V$ (resp. $V_{K}$).

If $M_{n}(K)$ denotes the ring of $n\times n$ matrices over $K$ and $\phi:K\rightarrow M_n(K)$ is a nonzero homomorphism,
then we denote by $K^n_\phi$ the two-sided vector space whose left action is the usual one whose right action is via $\phi$; that is,
$$
a \cdot(v_1,\dots, v_n)=(av_1,\dots,av_n),
\ \ \ (v_1,\dots, v_n)\cdot
a=(v_1,\dots,v_n)\phi(a).
$$
If $V$ is a two-sided vector space with left dimension equal to
$n$, then choosing a left basis for $V$ shows that $V\cong
{K^n_\phi}$ for some homomorphism $\phi:K\rightarrow M_n(K)$.

We say $V$ has {\it rank $n$} if $\operatorname{dim}_{K}({}_{K}V)
= \operatorname{dim}_{K}(V_{K})=n$.  Note that since $K/k$ is finite, if $\operatorname{dim}_{K}({}_{K}V)=n$ or $\operatorname{dim}_{K}(V_{K})=n$, then $V$ has rank $n$.  For the remainder of this section, we will assume two-sided vector spaces have finite rank $n$ unless otherwise stated.

\begin{lemma} \label{lemma.twosidedclass}
Suppose $\operatorname{char }k \neq 2$. If $V$ has rank two, then either
\begin{enumerate}
\item{} $V \cong K^{2}_{\phi}$ where $\phi(a)=\begin{pmatrix}
\sigma(a) & 0 \\ 0 & \sigma(a) \end{pmatrix}$ where $\sigma \in
\operatorname{Gal}(K/k)$,

\item{} $V \cong K^{2}_{\phi}$ where $\phi(a)=\begin{pmatrix}
\sigma(a) & 0 \\ 0 & \tau(a) \end{pmatrix}$, $\sigma(a), \tau(a)
\in \operatorname{Gal}(K/k)$, and $\tau \neq \sigma$, or

\item{} $V$ is simple.  In this case $V \cong K^{2}_{\phi}$ where $\phi(a)=\begin{pmatrix}
\gamma(a) &  \delta(a) \\ m \delta(a) & \gamma(a) \end{pmatrix}$ and where $\delta$ is a nonzero $(\gamma, \gamma)$-derivation of $K$, $m \in K$ is not a perfect square, and $\gamma(ab)=\gamma(a)\gamma(b)+m\delta(a)\delta(b)$.
\end{enumerate}
\end{lemma}

\begin{proof}
Let $\sigma \in \operatorname{Gal }(K/k)$.  We show that there are no nonzero $(\sigma, \sigma)$-derivations.
The result will then follow from \cite[Theorem 1.3]{P}.

Suppose $\delta$ is a $(\sigma,\sigma)$-derivation.  Then $\delta
\sigma^{-1}$ is an ordinary $k$-linear $K$-derivation, $\gamma$.
We claim $\gamma=0$. To this end, since $K$ is a finite extension
of a perfect field, it is separable.  Hence, $\Omega_{K/k}=0$.  On
the other hand, $\operatorname{Der }_{k}(K,K) \cong
\operatorname{Hom }_{K}(\Omega_{K/k},K)$.  Therefore, $\gamma$,
and hence, $\delta$, equals 0.
\end{proof}
We also need to recall (from \cite{vandenbergh}), the notion of left and right dual of a two-sided vector space.  The {\it right dual of $V$}, denoted $V^{*}$, is the set
$\operatorname{Hom}_{K}(V_{K},K)$ with action $(a \cdot \psi \cdot
b)(v)=a\psi(bv)$ for all $\psi \in
\operatorname{Hom}_{K}(V_{K},K)$ and $a,b \in K$.  We note that
$V^{*}$ is a $K\otimes_{k}K$-module since $V$ is.

The {\it left dual of $V$}, denoted ${}^{*}V$, is the set
$\operatorname{Hom}_{K}({}_{K}V,K)$ with action $(a \cdot \phi
\cdot b)(v)=b \phi(va)$ for all $\phi \in
\operatorname{Hom}_{K}({}_{K}V,K)$ and $a,b \in K$.  As above,
${}^{*}V$ is a $K \otimes_{k}K$-module.  This assignment extends to morphisms between two-sided vector spaces in the obvious way.

We set
$$
V^{i*}:=
\begin{cases}
V & \text{if $i=0$}, \\
(V^{(i-1)*})^{*} & \text{ if $i>0$}, \\
{}^{*}(V^{(i+1)*}) & \text{ if $i<0$}.
\end{cases}
$$

We recall that by \cite[Theorem 3.13]{hart}, $V \cong V^{**}$.  We will use this fact routinely in what follows.

If $V$ has rank $n$, then so does $V^{i*}$ for all $i$.  Therefore, $(-\otimes V^{i*},-\otimes V^{(i+1)*})$ has an adjoint pair structure for each $i$ by \cite{hart}[Proposition 3.7 and Corollary 3.5], and the Eilenberg-Watts Theorem implies that the unit of the pair $(-\otimes V^{i*},-\otimes V^{(i+1)*})$ induces a map of two-sided vector spaces
$K \rightarrow V^{i*} \otimes  V^{(i+1)*}$.  We sometimes denote the image of this map by $Q_{i}$, and we abuse notation by calling the map $\eta_{i}$ and referring to it as the unit map.

It is straightforward to check that if $\psi:V \rightarrow W$ is a morphism of two-sided vector spaces, than the left dual of $\psi$, ${}^{*}\psi$, equals the composition
$$
{}^{*}W \rightarrow {}^{*}V \otimes V \otimes {}^{*}W  \rightarrow {}^{*}V \otimes W \otimes {}^{*}W \rightarrow {}^{*}V
$$
whose left arrow is induced by a unit, whose right arrow is induced by a counit, and whose central arrow is induced by $\psi$.  A similar result holds for the right dual of $\psi$, $\psi^{*}$.  We will use this fact without comment in the sequel.

\subsection{Simple two-sided vector spaces} \label{section.simple}
Recall that Theorem \ref{thm.papp} describes simple two-sided vector spaces in terms of arithmetic data associated to the extension $K/k$.  In this section we describe arithmetic data associated to a twist of $V(\lambda)$ (Lemma \ref{lemma.autocomp}) and the arithmetic data associated to the tensor product $V(\lambda) \otimes V(\lambda)^{*}$ (Proposition \ref{prop.structure}).  The latter result will be used in the proof of Proposition \ref{prop.secondperiodic}.

\begin{lemma} \label{lemma.autocomp}
Suppose $\sigma, \tau \in \operatorname{Gal}(K/k)$.  Under the isomorphism given in Theorem \ref{thm.papp}, the $G$-orbit of $\overline{\sigma} \lambda \tau$ corresponds to the isomorphism class of the simple two-sided vector space $K_{\sigma} \otimes V(\lambda) \otimes K_{\tau}$.
\end{lemma}

\begin{proof}
By \cite[Lemma 4.4]{papp}, if $V(\lambda) \cong {}_{1}K^n_{\phi}$, then $K_{\sigma} \otimes V(\lambda) \otimes K_{\tau} \cong {}_{1}K^n_{\sigma \phi \tau}$ and the latter two-sided vector space is simple.  Therefore, the proof of Theorem \ref{thm.papp} implies that the class of embeddings corresponding to $K_{\sigma} \otimes V(\lambda) \otimes K_{\tau}$, equals the class of a common eigenvalue for the image of $\sigma \phi \tau:K \rightarrow M_{n}(K)$ considered as a subset of $M_{n}(\overline{K})$.  However, if $w \in \overline{K}^{n}$ is a common eigenvector for the image of $\phi$ with eigenvalue $\lambda$, then $\overline{\sigma}(w)$ is a common eigenvector for the image of $\sigma \phi \tau$ with eigenvalue $\overline{\sigma} \lambda \tau$.
\end{proof}

\begin{prop} \label{prop.structure}
Let $\lambda \in \operatorname{Emb}(K)$, let $\mu : =(\overline{\lambda})^{-1}$, suppose $\rho, \upsilon \in G$ are such that $\lambda^{G}=\{\lambda, \rho \lambda\}$ and $\mu^{G} = \{ \mu, \upsilon \mu\}$, and let $V:=V(\lambda)$ (so $V$ has rank 2 since $|\lambda^{G}|=2$).
\begin{enumerate}
\item{}
If $\overline{\lambda} \upsilon \mu(K) \subset K$ then $\gamma := \overline{\lambda} \upsilon \mu|_{K} \in \operatorname{Gal}(K/k)$ has order 2 and
$$
V \otimes V^{*} \cong K^{\oplus 2} \oplus K_{\gamma}^{\oplus 2},
$$

\item{}
if $\overline{\lambda} \upsilon \mu(K) \subset \rho \lambda(K)$, then $\delta := \mu \rho^{-1} \overline{\lambda} \upsilon \mu|_{K}$ is in $\operatorname{Gal}(K/k)$ and
$$
V \otimes V^{*} \cong K^{\oplus 2} \oplus (V \otimes K_{\delta}).
$$

\item{} Otherwise,
$$
V \otimes V^{*} \cong K^{\oplus 2} \oplus W
$$
where $W=V(\overline{\lambda}\upsilon \mu)$, and $W$ is not isomorphic to $V(\lambda \delta)$ for any $\delta \in \operatorname{Gal}(K/k)$.
\end{enumerate}
\end{prop}

\begin{proof}
By \cite[Theorem 3.13]{hart}, there is an isomorphism $V(\lambda)^{*} \rightarrow V(\mu)$, and by the proof of \cite[Theorem 4.7]{papp}, the tensor product $V(\lambda) \otimes V(\mu)$ is isomorphic to the direct sum of $K^{\oplus 2}$ with either $V(\overline{\lambda}\upsilon \mu)$ if $\overline{\lambda}\upsilon \mu$ has $G$-orbit size 2 or with $K_{\overline{\lambda}\upsilon \mu}^{\oplus 2}$ if $\overline{\lambda}\upsilon \mu$ has $G$-orbit size 1.

Suppose $\overline{\lambda} \upsilon \mu|_{K} \subset K$ so that $\overline{\lambda} \upsilon \mu|_{K} \in \operatorname{Gal}(K/k)$.  Then $\overline{\lambda} \upsilon \mu|_{K}$ is stable under the action of $G$, so has orbit size 1.  We first show that $\overline{\lambda} \upsilon \mu$ has order 2 in $\operatorname{Gal}(K/k)$.  To this end, we have $\overline{\lambda} \upsilon \mu|_{K}\overline{\lambda} \upsilon \mu|_{K}=\overline{\lambda} \upsilon^{2} \mu|_{K}$.  If $\upsilon^{2} \mu|_{K} = \upsilon \mu|_{K}$, then this would contradict the assumption on $\upsilon$.  Therefore, since $|\mu^{G}|=2$, we must have $\upsilon^{2} \mu|_{K} = \mu|_{K}$ so that $\overline{\lambda} \upsilon^{2} \mu|_{K}$ is the identity.  It follows that the order of $\overline{\lambda} \upsilon \mu|_{K}$ is at most 2.  To show the order is exactly 2, we suppose $|\overline{\lambda} \upsilon \mu|_{K}|=1$.  Then $\overline{\lambda} \upsilon \mu|_{K}=\overline{\lambda} \mu|_{K}$, which implies that $\upsilon \mu|_{K}=\mu|_{K}$.  This again contradicts the assumption on $\upsilon$ and the fact that $\overline{\lambda} \upsilon \mu$ has order 2 follows.  The second part of (1) follows from the proof of \cite[Theorem 4.7]{papp} in light of the fact that in this case, $\overline{\lambda}\upsilon \mu$ has orbit size 1.

Next, suppose that $\overline{\lambda} \upsilon \mu(K) \subset \rho \lambda(K)$.  Then $\delta := \mu \rho^{-1} \overline{\lambda} \upsilon \mu|_{K}$ is an element of $\operatorname{Gal}(K/k)$.  By definition of $\delta$, $\lambda \delta = \rho^{-1} \overline{\lambda} \upsilon \mu_{K}$, and thus
\begin{equation} \label{eqn.deltarel}
\rho \lambda \delta = \overline{\lambda} \upsilon \mu|_{K}.
\end{equation}
Now, on the one hand, $V \otimes K_{\delta}$ corresponds, via Lemma \ref{lemma.autocomp}, to the $G$-orbit of $\lambda \delta$.  This orbit is $\{\lambda \delta, \rho \lambda \delta\}$.  On the other hand, the nontrivial summand in $V \otimes V^{*}$ corresponds to the $G$-orbit of $\overline{\lambda} \upsilon \mu$.  Since these orbits are equal by (\ref{eqn.deltarel}), the second part of (2) follows.

In the final case, we first note that $\overline{\lambda} \upsilon \mu (K)$ is not contained in $K$, so that the $G$-orbit size of $\overline{\lambda} \upsilon \mu$ is two.  We next note that if $\overline{\lambda} \upsilon \mu (K) = \lambda(K)$, then $\upsilon \mu \in \operatorname{Gal}(K/k)$, which is a contradiction.  Therefore, $\overline{\lambda} \upsilon \mu(K)$ is not equal to $\lambda(K)$ or $\rho \lambda (K)$.  It follows that $\overline{\lambda} \upsilon \mu|_{K} \neq \lambda \delta$ and $\overline{\lambda} \upsilon \mu|_{K} \neq \rho \lambda \delta$ for any $\delta \in \operatorname{Gal}(K/k)$.
\end{proof}

\subsection{Noncommutative symmetric algebras}

In this section we recall (from \cite{vandenbergh}) the definition of the noncommutative symmetric algebra of a rank $n$ two-sided vector space $V$.  The {\it noncommutative symmetric algebra generated by $V$}, denoted $\mathbb{S}(V)$, is the $\mathbb{Z}$-algebra (see \cite[Section 2]{quadrics} for a definition of $\mathbb{Z}$-algebra) $\underset{i,j \in
\mathbb{Z}}{\oplus}A_{ij}$ with components defined as follows:
\begin{itemize}
\item{} $A_{ij}=0$ if $i>j$.

\item{} $A_{ii}=K$.

\item{} $A_{i,i+1}=V^{i*}$.
\end{itemize}
In order to define $A_{ij}$ for $j>i+1$, we introduce some notation: we define $T_{i,i+1} := A_{i,i+1}$, and, for $j>i+1$, we define
$$
T_{ij} := A_{i,i+1} \otimes A_{i+1,i+2} \otimes \cdots \otimes A_{j-1,j}.
$$
We let $R_{i,i+1}:= 0$, $R_{i,i+2}:=Q_{i}$,
$$
R_{i,i+3}:=Q_{i} \otimes V^{(i+2)*}+V^{i*} \otimes Q_{i+1},
$$
and, for $j>i+3$, we let
$$
R_{ij} := Q_{i} \otimes T_{i+2,j}+T_{i,i+1}\otimes Q_{i+1} \otimes T_{i+3,j}+\cdots + T_{i,j-2} \otimes Q_{j-2}.
$$

\begin{itemize}
\item{} For $j>i+1$, we define $A_{ij}$ as the quotient $T_{ij}/R_{ij}$.
\end{itemize}
Multiplication in $\mathbb{S}(V)$ is defined as follows:
\begin{itemize}

\item{} if $x \in A_{ij}$, $y \in A_{lk}$ and $j \neq l$, then $xy=0$,

\item{} if $x \in A_{ij}$ and $y \in A_{jk}$, with either $i=j$ or $j=k$, then $xy$ is induced by the usual scalar action,

\item{}  otherwise, if $i<j<k$, we have
\begin{eqnarray*}
A_{ij} \otimes A_{jk} & = & \frac{T_{ij}}{R_{ij}} \otimes \frac{T_{jk}}{R_{jk}}\\
& \cong & \frac{T_{ik}}{R_{ij}\otimes T_{jk}+T_{ij} \otimes
R_{jk}}.
\end{eqnarray*}
Since $R_{ij} \otimes T_{jk}+T_{ij} \otimes R_{jk}$ is a submodule of $R_{ik}$,
there is thus an epi $\mu_{ijk}:A_{ij} \otimes A_{jk} \rightarrow
A_{ik}$.
\end{itemize}

\begin{lemma} \label{lemma.periodic}
For $i,j \in \mathbb{Z}$ with $j>i$, there is an isomorphism $V^{i*} \otimes \cdots \otimes V^{j-1*} \rightarrow V^{i+2*} \otimes \cdots \otimes V^{j+1*}$ compatible with the relations in $\mathbb{S}(V)$.
\end{lemma}

\begin{proof}
For each $i\in \mathbb{Z}$ we construct an isomorphism $\phi_{i}:V^{i*} \rightarrow V^{i+2*}$.  The desired isomorphism will be the tensor product of these.

We begin by inductively constructing isomorphisms $\phi_{n}: V^{n*} \rightarrow V^{n+2*}$ for $n \geq 0$.  There exists an isomorphism $\phi_{0}: V \rightarrow V^{2*}$ by \cite[Theorem 3.13]{hart}.  Now, suppose there is an isomorphism $\phi_{i}:V^{i*} \rightarrow V^{i+2*}$.  This induces an isomorphism of functors
$$
-\otimes  V^{i*} \rightarrow -\otimes  V^{i+2*},
$$
and we abuse notation by denoting it by $\phi_{i}$.  On the other hand, by \cite[Proposition 3.7]{hart}, there is an adjunction with left and right adjoint $-\otimes  V^{l*}$ and $-\otimes  V^{l+1*}$ respectively, for $l=i$ and $l=i+2$.  It follows that there exists a unique natural equivalence $\psi:-\otimes V^{i+3*} \rightarrow -\otimes V^{i+1*}$ between right adjoints such that $(\phi_{i}, \psi)$ is a conjugate pair (see \cite[Theorem 2, p.100]{cats} for a definition of conjugate pair and a proof of this fact).  Therefore, if we again abuse notation by letting $\psi:V^{i+3*} \rightarrow V^{i+1*}$ denote the isomorphism corresponding, via the Eilenberg-Watts Theorem, to the natural transformation $\psi$, then we set $\phi_{i+1}=\psi^{-1}$.

To construct $\phi_{i}$ for $i<0$ we proceed inductively as above using left adjoints instead of right adjoints to define $\psi$.

It remains to prove that the isomorphism is compatible with the relations in $\mathbb{S}(V)$.  To this end, by the definition of conjugate pair in \cite[Theorem 2, p.100]{cats}, the diagram
$$
\begin{CD}
K & \rightarrow & V^{i*} \otimes  V^{i+1*} \\
@VVV @VV{\phi_{i} \otimes  V^{i+1*}}V \\
V^{i+2*} \otimes  V^{i+3*} & \underset{V^{i+2**} \otimes \phi_{i+1}^{-1}}{\rightarrow} & V^{i+2*} \otimes V^{i+1*}
\end{CD}
$$
whose unlabeled arrows are induced by counits of the appropriate adjoint pair, commutes.  The lemma follows from this.
\end{proof}

Following \cite[Section 2]{quadrics}, if $A$ is a $\mathbb{Z}$-algebra and $i \in \mathbb{Z}$, we let $A(i)$ denote the $\mathbb{Z}$-algebra such that $A(i)_{jl}=A_{i+j,i+l}$ and with multiplication induced from that on $A$.  As in the proof of Lemma \ref{lemma.periodic}, the identity map $V^{i*} \rightarrow V^{i*}$ induces an isomorphism of $\mathbb{Z}$-algebras over $K$,
\begin{equation} \label{eqn.canonicalz}
\mathbb{S}(V)(i) \rightarrow \mathbb{S}(V^{i*}).
\end{equation}
Let $J \subset \mathbb{Z}$ denote the set of all even numbers, and let $B$ denote the $J$-Veronese of $\mathbb{S}(V)$ \cite[Section 2]{quadrics}.  The following result is implicit in \cite{vandenbergh}, but since it is not stated explicitly there we include it here for the readers convenience.

\begin{cor} \label{cor.periodic}
The algebra $B$ is $2$-periodic, that is, there is an isomorphism of $\mathbb{Z}$-algebras $B \cong B(2)$.
\end{cor}

\begin{proof}
It follows from Lemma \ref{lemma.periodic} that there are $K-K$-bimodule isomorphisms $\phi_{ij}:B_{ij} \rightarrow B_{i+2,j+2}$ for all $i,j \in 2 \mathbb{Z}$.  The fact that the $\phi_{ij}$ are compatible with multiplication in $B$, follows from a routine but tedious verification, which we omit.
\end{proof}

For the remainder of this section, we will work towards proofs of Corollary \ref{cor.gen1} and Corollary \ref{cor.gen2}, which will be used to prove Proposition \ref{prop.newnewfactor}.  We now introduce some notation we will utilize in this section:  if ${\sf E}$ is a category, we let $I_{\sf E}$ denote the identity functor on ${\sf E}$.
\begin{lemma} \label{lemma.categories}
Let ${\sf C}$ and ${\sf D}$ be categories, let $F, F':{\sf C} \rightarrow {\sf D}$ and $G, G':{\sf D} \rightarrow {\sf C}$ denote functors, and let $(F,G,\eta,\epsilon)$, $(F',G',\eta',\epsilon')$ denote adjunctions.  Suppose that $\alpha:I_{\sf C} \rightarrow I_{\sf C}$, $\beta:I_{\sf D} \rightarrow I_{\sf D}$, $\gamma_{i}:F \rightarrow F'$, and $\delta_{i}:G \rightarrow G'$ are isomorphisms for $i=1,2$, and consider the diagrams
$$
\begin{CD}
I_{\sf C} & \overset{\eta}{\rightarrow} & GF & &\hskip .5in & & FG & \overset{\epsilon}{\rightarrow} & I_{\sf D} \\
@V{\alpha}VV @VV{\delta_{1} \gamma_{i}}V &  \hskip .5in & @V{\gamma_{1} \delta_{i}}VV @VV{\beta}V \\
I_{\sf C} & \underset{\eta'}{\rightarrow} & G'F' & & \hskip .5in & & F'G' & \underset{\epsilon'}{\rightarrow} & I_{\sf D}
\end{CD}
$$
If the left diagram commutes for $i=1,2$, then $\gamma_{1}=\gamma_{2}$, and if the right diagram commutes for $i=1,2$, then $\delta_{1}=\delta_{2}$.
\end{lemma}

\begin{proof}
We prove the first statement.  The proof of the second is dual to that of the first, and is omitted.

The universal property of $\eta$ implies that for each object $C$ in ${\sf C}$, there exists a unique $g_{C}:FC \rightarrow F'C$ such that the diagram
$$
\begin{CD}
C & \overset{\eta_{C}}{\rightarrow} & GFC \\
@V{\eta_{C}' \alpha_{C}}VV @VV{Gg_{C}}V \\
G'F'C & \underset{{\delta^{-1}_{1}}_{F'C}}{\rightarrow} & GF'C
\end{CD}
$$
commutes.  Therefore, $\gamma_{iC}=g_{C}$ for $i=1,2$ as desired.
\end{proof}

\begin{cor} \label{cor.gen1}
Suppose that $\eta:K \rightarrow {}^{*}V \otimes V$ and $\eta':K \rightarrow {}^{*}W \otimes W$ are unit maps, and that $\psi:V \rightarrow W$ is an isomorphism.  Let $0 \neq a \in K$ and let ${}_{a}\mu:K \rightarrow K$ denote left-multiplication by $a$.  If $\phi:{}^{*}V \rightarrow {}^{*}W$ is an isomorphism such that
$$
\begin{CD}
K & \overset{\eta}{\rightarrow} & {}^{*}V \otimes V \\
@V{{}_{a}\mu}VV @VV{\phi \otimes \psi}V \\
K & \underset{\eta'}{\rightarrow} & {}^{*}W \otimes W
\end{CD}
$$
commutes, then $\phi={}_{a}\mu ({}^{*}\psi)^{-1}$.
\end{cor}

\begin{proof}
By \cite[Theorem 1.2]{ns} and Lemma \ref{lemma.categories}, it suffices to show that $\phi={}_{a}\mu ({}^{*}\psi)^{-1}$ makes the diagram commute.  This fact follows from \cite[Corollary 6.7]{duality}.
\end{proof}

We now work towards proving Corollary \ref{cor.gen2}.  To this end, we recall that there is a canonical isomorphism
$$
V^{**} \otimes V^{*} \rightarrow (V \otimes V^{*})^{*}
$$
defined as follows: by \cite[Proposition 3.7]{hart}, there are adjunctions
$$
(-\otimes V,-\otimes V^{*}, \eta_{0},\epsilon_{0}),
$$
$$(-\otimes V^{*},-\otimes V^{**}, \eta_{1},\epsilon_{1}),
$$
and
$$
(-\otimes (V \otimes V^{*}),-\otimes (V \otimes V^{*})^{*},\eta_{01},\epsilon_{01}).
$$
By \cite[Theorem 1, p. 103]{cats}, we may compose the first two adjunctions to get another adjunction.  Thus, by the uniqueness of right adjoints, there exists a unique isomorphism of functors $\psi:-\otimes (V^{**} \otimes V^{*}) \rightarrow - \otimes (V \otimes V^{*})^{*}$ such that the diagram
\begin{equation} \label{eqn.adjoint}
\begin{CD}
-\otimes (V^{**} \otimes V^{*}) \otimes (V \otimes V^{*}) & \rightarrow & -\otimes (V \otimes V^{*})^{*} \otimes (V \otimes V^{*}) \\
@VVV @VV{\epsilon_{01}}V \\
I_{{\sf Mod }K} & = & I_{{\sf Mod }K}.
\end{CD}
\end{equation}
whose top horizontal is induced by $\psi$ and whose left vertical is induced by $\epsilon_{0}$ and $\epsilon_{1}$, commutes.  The canonical isomorphism $V^{**} \otimes V^{*} \rightarrow (V \otimes V^{*})^{*}$ is the isomorphism corresponding to $\psi$ via the Eilenberg-Watts Theorem.

\begin{lemma} \label{lemma.mess}
Retain the notation in the preceding paragraph and let $\epsilon:V^{**} \otimes V^{*} \rightarrow K$ denote the composition
$$
V^{**} \otimes V^{*} \rightarrow (V \otimes V^{*})^{*} \overset{\eta^{*}_{0}}{\rightarrow} K^{*} \rightarrow K
$$
whose first and last maps are canonical.  Then
$$
(-\otimes V^{*},-\otimes V^{**}, \eta_{1},\epsilon)
$$
is an adjunction.
\end{lemma}

\begin{proof}
It suffices to show that the compositions
$$
V^{**} \rightarrow V^{**} \otimes K \overset{V^{**} \otimes \eta_{1}}{\rightarrow} V^{**} \otimes V^{*} \otimes V^{**} \overset{\epsilon \otimes V^{**}}{\rightarrow} K \otimes V^{**} \rightarrow V^{**}
$$
and
$$
V^{*} \rightarrow K \otimes V^{*} \overset{\eta_{1}\otimes V^{*}}{\rightarrow} V^{*} \otimes V^{**} \otimes V^{*} \overset{V^{*} \otimes \epsilon}{\rightarrow} V^{*} \otimes K \rightarrow V^{*}
$$
whose unlabeled morphisms are canonical, are the identity.  To this end, after expanding the compositions using the definition of $\epsilon$ and $\eta_{0}^{*}$, one utilizes the commutativity of (\ref{eqn.adjoint}) as well as the fact that the compositions
$$
V^{**} \rightarrow V^{**} \otimes K \overset{V^{**} \otimes \eta_{i}}{\rightarrow} V^{**} \otimes V^{*} \otimes V^{**} \overset{\epsilon_{i} \otimes V^{**}}{\rightarrow} K \otimes V^{**} \rightarrow V^{**}
$$
and
$$
V^{*} \rightarrow K \otimes V^{*} \overset{\eta_{i}\otimes V^{*}}{\rightarrow} V^{*} \otimes V^{**} \otimes V^{*} \overset{V^{*}\otimes \epsilon_{i}}{\rightarrow} V^{*} \otimes K \rightarrow V^{*}
$$
are the identity for $i=0,1$.  The straightforward but tedious details are left to the reader.
\end{proof}

\begin{cor} \label{cor.gen2}
Suppose that $\eta:K \rightarrow {}^{*}V \otimes V$ and $\eta':K \rightarrow {}^{*}W \otimes W$ are unit maps and that $\psi:V \rightarrow W$ is an isomorphism.  Let $0 \neq a \in K$ and let $\mu_{a}:K \rightarrow K$ denote right multiplication by $a$.  If $\phi:V^{*} \rightarrow W^{*}$ is an isomorphism such that
\begin{equation} \label{eqn.corelation}
\begin{CD}
K & \overset{\eta}{\rightarrow} & V \otimes V^{*} \\
@V{\mu_{a}}VV @VV{\psi \otimes \phi}V \\
K & \underset{\eta'}{\rightarrow} & W \otimes W^{*}
\end{CD}
\end{equation}
commutes, then $\phi=(\psi^{*})^{-1} \mu_{a}$.
\end{cor}

\begin{proof}
By taking the right dual of (\ref{eqn.corelation}), we get a diagram
\begin{equation} \label{eqn.newcorelation}
\begin{CD}
V^{**} \otimes V^{*} & \rightarrow & (V \otimes V^{*})^{*} & \overset{\eta^{*}}{\rightarrow} & K^{*} & \rightarrow & K \\
@A{\phi^{*}\otimes \psi^{*}}AA @A{(\psi \otimes \phi)^{*}}AA  @AA{\mu_{a}^{*}}A @AA{{}_{a}\mu}A \\
W^{**} \otimes W^{*} & \rightarrow & (W \otimes W^{*})^{*} & \underset{{\eta'}^{*}}{\rightarrow} & K^{*} & \rightarrow & K
\end{CD}
\end{equation}
whose unlabeled maps are canonical.  The diagram commutes by the functoriality of $(-)^{*}$ and by \cite[Corollary 6.3]{duality}, and the compositions of the horizontal maps equal the map $\epsilon$ defined in the statement of Lemma \ref{lemma.mess}.  On the one hand, if we replace $\phi^{*}$ by ${}_{a}\mu (\psi^{**})^{-1}$, the outer circuit of (\ref{eqn.newcorelation}) commutes by \cite[Corollary 6.7]{duality}.  On the other hand, by \cite[Theorem 1.2]{ns}, and Lemma \ref{lemma.categories}, this choice of $\phi^{*}$ is unique making the outer circuit of (\ref{eqn.newcorelation}) commute.  Finally, since $V^{*}$ and $W^{*}$ are finite rank, one can explicitly check, using the proof of \cite[Proposition 3.7]{hart}, that right duality is faithful on isomorphisms $V^{*} \rightarrow W^{*}$.  The result now follows from the fact that $\mu_{a}^{*}={}_{a}\mu$.
\end{proof}

\section{Vector bundles over $\mathbb{P}(V)$} \label{section.groth}
Throughout this section, we let $A$ denote the $\mathbb{Z}$-algebra $\mathbb{S}(V)$ where $V$ is rank 2, we let $J \subset \mathbb{Z}$ denote the subset of even integers, and we let $B$ denote the $J$-Veronese of $A$.  We recall that by \cite[Section 6.3]{vandenbergh}, the category ${\sf Gr }A$ is locally noetherian.  We define a $\mathbb{Z}$-graded ring $C$ by letting
$$
C_{i} := B_{0,2i}
$$
and by letting multiplication in $C$ be defined by the composition
\begin{eqnarray*}
C_{i} \otimes C_{j} & = & B_{0,2i} \otimes B_{0, 2j} \\
& \cong & B_{0, 2i} \otimes B_{2i, 2i+2j} \\
& \rightarrow & B_{0, 2i+2j} \\
& = & C_{i+j}
\end{eqnarray*}
whose second map is from Corollary \ref{cor.periodic}, and whose third map is multiplication in $B$.

In this section, we prove that $C$ is a domain (Corollary \ref{cor.domain}), we prove that $\mathbb{P}(V)$ is an integral noncommutative space (Corollary \ref{cor.integral}), and we classify vector bundles over $\mathbb{P}(V)$ (Corollary \ref{cor.grothendieck}).  This last result generalizes the main result in \cite{groth} when the base field is perfect.

Our proof that $C$ is a domain follows from a variant of the proof of \cite[Lemma 3.15]{ATV2}, and is thus homological in nature.  For this reason, we begin this section by developing some basic homological results for $A$.  Many of these results (and their proofs) are motivated by the intuition that $A$ behaves as if it were a noetherian regular algebra (in the sense of \cite{ATV2}) of global dimension $2$.

Since $K/k$ is finite and $k$ is perfect, $\operatorname{Spec }K$ is smooth of finite type over $k$ so that we may use most of the results of \cite[Sections 4-7]{vandenbergh}, all of \cite[Sections 1-4]{duality} and most of \cite[Section 13-15]{chan} (in the last reference, the global hypothesis that the base field is algebraically closed is not needed for the results we will use).  In particular, we will employ internal hom and tensor functors on ${\sf Gr }A$, as well as their derived functors.  More specific references will be given when these functors are invoked.

\subsection{Homological preliminaries}
We recall \cite[Section 2.2]{duality} that $(e_{k}A)_{\geq k+n}$ is defined to be the sum of $K-K$-bimodules $\underset{i \geq 0}{\oplus}e_{k}{A}_{k+n+i}$.  We define $A_{\geq n} := \underset{k}{\oplus} (e_{k}A)_{\geq k+n}$ with its natural $A-A$-bimodule structure and we let $A_{0}$ denote the $A-A$-bimodule $A/A_{\geq 1}$.

An object in ${\sf Gr }A$ is called {\it free} if it is isomorphic to a module of the form $\oplus_{i} e_{i}A$.

Let $M \in {\sf gr }A$.  We call a resolution
$$
\cdots F_{i} \overset{f_{i}}{\rightarrow} F_{i-1} \overset{f_{i-1}}{\rightarrow} \cdots \rightarrow F_{0} \overset{f_{0}}{\rightarrow} M \rightarrow 0
$$
{\it minimal} if the $\operatorname{ker }f_{i} \subset F_{i}A_{\geq 1}$ and if $F_{i}$ is free and noetherian.

The following result allows us to construct minimal projective resolutions of noetherian $A$-modules, following \cite[Chapter 10, Lemma 2.4]{smith}:
\begin{lemma} \label{lemma.minimal}
Let $M$ be a noetherian $A$-module.  Then there exists a finite-dimensional $K$-subspace $U \subset M$ with a homogeneous basis such that $M=U \oplus UA_{\geq 1}$.  It follows that there exists an epimorphism $f$ from a noetherian free module $P$ to $M$ with the property that $\operatorname{ker }f \subset PA_{\geq 1}$.
\end{lemma}

\begin{proof}
We begin with the proof of the first part of the lemma.  Suppose $M$ is left bounded by degree $d$.  We let $U_{d}:= M_{d}$.  If $M=U_{d} \oplus U_{d}A_{\geq 1}$ we are done.  Otherwise, for some $n>0$ we may define nonzero vector spaces $U_{d+1}, \ldots, U_{d+n}$ by
$$
U_{d+i} := \operatorname{Span}_{K}\{u \in M_{d+i}| u \notin (U_{d} \oplus \cdots \oplus U_{d+i-1})A_{\geq 1}\}.
$$
Since $M$ is noetherian, there must exist an $m>0$ such that
$$
M=U_{d} \oplus \cdots \oplus U_{d+m} + (U_{d} \oplus \cdots \oplus U_{d+m})A_{\geq 1}.
$$
Furthermore, by degree considerations and the definition of the subspaces $U_{d+i}$,
$$
(U_{d} \oplus \cdots \oplus U_{d+m}) \cap (U_{d} \oplus \cdots \oplus U_{d+m})A_{\geq 1}=0.
$$
The first part of the lemma follows with $U:= U_{d} \oplus \cdots \oplus U_{d+m}$.

For the second part of the lemma, we let $P:= \oplus_{i=0}^{m} U_{d+i} \otimes  e_{d+i}A$ and we let $f:P \rightarrow M$ be induced by the $A$-module action on $M$.  The fact that $f$ is epic follows from the fact that $M$ is generated by $U$.

Finally, we prove $\operatorname{ker }f \subset PA_{\geq 1}$.  First, suppose $0 \neq p \in P$ is homogeneous of degree $d+n$ for some $0 \leq n \leq m$, and $f(p)=0$.  Then $n>0$ and
$$
p=\sum_{i_{0}}u_{i_{0}}\otimes a_{i_{0}} + \cdots + \sum_{i_{n}}u_{i_{n}} \otimes a_{i_{n}}
$$
where $u_{i_{l}} \in U_{d+l}$ and $a_{i_{l}} \in A_{d+l, d+n}$.  Since $f(p)=0$,
\begin{equation} \label{eqn.minorsum}
\sum_{i_{0}}u_{i_{0}}a_{i_{0}} + \cdots + \sum_{i_{n}}u_{i_{n}}a_{i_{n}}=0.
\end{equation}
Since $a_{i_{n}} \in A_{d+n,d+n}$, if the last sum of (\ref{eqn.minorsum}) were nonzero, then some element of $U_{d+n}$ would be in $(U_{d}+\cdots+U_{d+n-1})A_{\geq 1}$ which contradicts the definition of $U_{d+n}$.  Therefore, the last sum in (\ref{eqn.minorsum}) is zero and $p \in P A_{\geq 1}$.

Now suppose $p \in \operatorname{ker }f$ and $p$ is homogeneous of degree $\geq d+m+1$.  Then by degree considerations, we must have $p \in PA_{\geq 1}$.  The result follows.
\end{proof}

In the next result we employ the functors $\underline{\otimes}_{A}$ and $\underline{Tor}_{i}^{A}(-,-)$ defined and studied in \cite[Section 14]{chan}.  For the readers convenience, we recall the definitions.  Let $M$ denote a graded right $A$-module and let $N$ denote a graded $A-K$-bimodule, i.e. $N=\oplus_{m \in \mathbb{Z}}N_{m}$, where $N_{m}$ are $k$-central $K-K$-bimodules, and there are multiplication maps $A_{m,n} \otimes_{K}N_{n} \rightarrow N_{m}$ satisfying the usual associativity and unit axiom.  We let

$$
M \underline{\otimes}_{A} N := \mbox{coker}\ \bigl(\bigoplus_{l,m} M_l \otimes_K A_{l,m} \otimes_K N_m \xrightarrow{\mu \otimes 1 - 1 \otimes \mu} \bigoplus_n M_n \otimes_K N_n \bigr),
$$
where $\mu$ denotes the appropriate multiplication.  A graded $A-K$-module $N$ is called {\it internally flat} if $-\underline{\otimes}_{A}N$ is exact.  The bimodule $N$ is said to have {\it finite internal flat dimension} if it has a finite resolution by internally flat bimodules.  The notions of internally flat right $A$-module and finite internal flat dimension of an $A$-module are defined similarly.  Finally, for a graded $A-K$-module $N$ of finite internal flat dimension, the derived functors $\underline{Tor}_{i}^{A}(-,N)$ are defined using internally flat resolutions in the first component.

\begin{prop} \label{prop.global}
If $F$ is a noetherian free module and $i:FA_{\geq 1} \rightarrow F$ is inclusion, then $i \underline{\otimes}_{A}A_{0}$ vanishes.  It follows that if $M$ is an object in ${\sf gr }A$ with a minimal resolution
$$
\cdots F_{i} \overset{f_{i}}{\rightarrow} F_{i-1} \overset{f_{i-1}}{\rightarrow} \cdots \rightarrow F_{0} \overset{f_{0}}{\rightarrow} M \rightarrow 0
$$
by free noetherian objects $F_{i}$, and if $F_{j} \neq 0$, for some $j$, then $\underline{Tor}_{j}^{A}(M,A_{0}) \neq 0$.  Therefore, the length of a minimal resolution of a noetherian module is at most $2$.
\end{prop}

\begin{proof}
To prove the first assertion, we first note that $e_{m}A \underline{\otimes}_{A}A_{0} \cong A_{mm}=K$ by \cite[Proposition 3.5]{duality}.  On the other hand, by the right exactness of $-\underline{\otimes}_{A}A_{0}$, and by \cite[Proposition 15.2]{chan}, the module $(e_{m}A/e_{m}A_{\geq m+1}) \underline{\otimes}_{A} A_{0}$ is also isomorphic to $K$ as a $K$-module.  Thus, if
$$
i:e_{m}A_{\geq m+1} \rightarrow e_{m}A
$$
is inclusion, then $i \underline{\otimes}_{A}A_{0}$ vanishes.  Therefore, if $F$ is a noetherian free module and $i:FA_{\geq 1} \rightarrow F$ is inclusion, then it follows that $i \underline{\otimes}_{A}A_{0}=0$.

Next, we prove the second assertion.  Since $f_{j}$ factors through the inclusion $F_{j-1}A_{\geq 1} \rightarrow F_{j-1}$, the first assertion implies that $f_{j} \underline{\otimes}_{A} A_{0}=0$ for $j \geq 0$.  It follows that $\underline{Tor}_{j}^{A}(M,A_{0}) \cong F_{j} \underline{\otimes}_{A} A_{0}$.  Therefore, if $F_{j} \neq 0$, then $\underline{Tor}_{j}^{A}(M,A_{0})\neq 0$ by \cite[Lemma 15.2]{chan}.

To prove the final assertion, we recall from \cite[Section 14]{chan} that internal tor can also be computed using internally flat resolutions of $A_{0}$.  Since $A_{0}$ has an internally flat resolution of length 2 \cite[Propostion 14.2]{chan}, the length of a minimal resolution of a noetherian module is at most 2.
\end{proof}

\begin{corollary} \label{cor.free}
Suppose $P$ is a noetherian projective object in ${\sf Gr }A$.  Then $P$ is free.
\end{corollary}

\begin{proof}
By Lemma \ref{lemma.minimal}, there exists an epimorphism $f$ from a noetherian free module $F$ to $P$ with the property that $\operatorname{ker }f \subset FA_{\geq 1}$.  Since $P$ is projective, $f$ splits.  Since $\operatorname{ker }f$ is a summand of $F$, if $i:\operatorname{ker }f \rightarrow F$ denotes inclusion, then $i \underline{\otimes}_{A} A_{0}$ is injective.  By the first part of Proposition \ref{prop.global}, $i \underline{\otimes}_{A} A_{0}=0$, so that $(\operatorname{ker }f) \underline{\otimes}_{A} A_{0}=0$.  Therefore, inclusion of $P$ in $F$ induces a surjection $P \underline{\otimes}_{A} A_{0} \rightarrow F \underline{\otimes}_{A} A_{0}$.  It thus follows from \cite[Lemma 15.2]{chan} that inclusion is a surjection.
\end{proof}
The following is inspired by \cite[Proposition 2.40(i)]{ATV2}.
\begin{corollary} \label{cor.global}
Suppose $M$ is a graded noetherian right $A$-module and there is an exact sequence
\begin{equation} \label{eqn.mm}
0 \rightarrow K \rightarrow F_{1} \rightarrow F_{0} \overset{f_{0}}{\rightarrow}  M \rightarrow 0
\end{equation}
with $F_{0}$ and $F_{1}$ noetherian and free, $\operatorname{ker }f_{0} \subset F_{0}A_{\geq 1}$, and $K \subset F_{1}A_{\geq 1}$.  Then $K$ is free.
\end{corollary}

\begin{proof}
If $K$ were not free, it would have a minimal resolution of length at least 1.  Attaching this resolution to (\ref{eqn.mm}) would yield a minimal resolution of $M$ of length greater than 2, violating Proposition \ref{prop.global}.
\end{proof}

In the next result, we use the internal hom functors and their derived functors introduced in \cite[Section 3.2 and Section 4.1]{duality}.  The statement and proof are motivated by \cite[Proposition 2.46(i)]{ATV2}.

\begin{lemma} \label{lemma.homological}
Suppose $N$ and $P$ are noetherian free $A$-modules and $f:N \rightarrow P$ is a morphism whose image is contained in $PA_{\geq 1}$.  Then the induced morphism, $\underline{Ext}_{{\sf Gr }A}^{2}(A_{0},f)=0$.  Therefore, if $M$ is noetherian, $\underline{Hom}_{{\sf Gr}A}(A_{0},M)=0$ if and only if the length of a minimal resolution of $M$ is at most $1$.
\end{lemma}

\begin{proof}
Since the functor $\underline{Ext}_{{\sf Gr }A}^{2}(A_{0},-)$ is additive, it suffices to prove the result when $N=e_{i}A$ and $P=e_{j}A$.  By \cite[Theorem 4.4]{duality}, $\underline{Ext}_{{\sf Gr }A}^{2}(A_{0},e_{k}A)$ is concentrated in degree $k-2$.  Therefore, the result holds when $i \neq j$.  Now, suppose $i=j$.  Then the only graded right $A$-module homomorphisms $f: N \rightarrow P$ are induced by scalar multiplication.  Therefore, if $f \neq 0$, the image of $f$ is not in $PA_{\geq 1}$.  The first result follows.

To prove the second result, we note that by \cite[Theorem 4.4]{duality}, Proposition \ref{prop.global}, and the first part of this lemma, the proof of \cite[Proposition 2.46i]{ATV2} works in our context, and establishes the second part of the lemma.
\end{proof}

\subsection{Proof that $\mathbb{P}(V)$ is integral}
By \cite[Theorem 6.1.2]{vandenbergh}, we know that
$$
\operatorname{dim }_{K}(e_{j}A_{j+i})_{K}=i+1.
$$
It follows from this, Lemma \ref{lemma.minimal} and Proposition \ref{prop.global} that if $M \in {\sf gr }A$, the function $f(n):= \operatorname{dim}_{K}M_{n}$ is eventually of the form $cn+d$, with $c, d \in \mathbb{Z}$ and $c \geq 0$. If $M \neq 0$, we let $\operatorname{dim }M=1$ if $c>0$ and $0$ otherwise.  If $M=0$ define $\operatorname{dim }M = -1$.  Furthermore, if $M \in {\sf Gr }A$ is not noetherian, let $\operatorname{dim }M = \operatorname{sup }\{\operatorname{dim }N | N \leq M \mbox{ is noetherian}\}$.  It is now straightforward to check that $\operatorname{dim }(-)$ is a dimension function on ${\sf Gr }A$ in the sense of \cite[Definition 3.1]{papp2}.

In the proof of Theorem \ref{thm.domain}, below, and in the proof that ${\sf Proj }A$ is integral (completed in Corollary \ref{cor.integral}), we will need the following

\begin{prop}
The category of graded left $A$-modules is locally noetherian.
\end{prop}
\noindent{\it Proof sketch.}  As one can check, the statements of \cite[Lemma 3.2.1 and Lemma 3.2.2]{vandenbergh} are true if the term {\it noetherian} is interpreted as meaning that the category of graded {\it left} modules over the algebra in question is locally noetherian.  Furthermore, there is a graded left module analogue of \cite[Definition 3.3.2]{vandenbergh}, and the corresponding versions of \cite[Theorem 3.3.3 and Theorem 5.2.1]{vandenbergh} hold in our context.  Therefore, again in our context, it is straightforward to check that if $\mathcal{D}$ is the $\mathbb{Z}$-algebra in \cite[Step 9 of Section 6.3]{vandenbergh}, then the category of graded left $\mathcal{D}$-modules is locally noetherian.  The result then follows from the graded left module version of \cite[Lemma 3.2.2]{vandenbergh} mentioned in the first line of this proof. \hfill $\Box$
\newline

We call a graded left module {\it free} if it is isomorphic to a module of the form $\oplus_{i}Ae_{i}$.  There are graded left module analogues of Lemma \ref{lemma.minimal} and Proposition \ref{prop.global}, which allows one to deduce that every noetherian object in the category of graded left $A$-modules has a minimal resolution of length at most 2.  It follows that if $M$ is a noetherian graded left $A$-module, the function $f(n) := \operatorname{dim}_{K}M_{-n}$ is eventually of the form $cn+d$, with $c, d \in \mathbb{Z}$ and $c \geq 0$.  The dimension of a graded left module is defined analogously to the dimension of a graded right module.

The proof of the following result is an adaptation of the proof of \cite[Lemma 3.15]{ATV2} to the present context.

\begin{thm} \label{thm.domain}
If $x,y \in A$ are homogeneous elements such that $xy=0$, then either $x=0$ or $y=0$.
\end{thm}

\begin{proof}
We note that it suffices to prove that if $y \in A_{jl}$ has the property that there exists a nonzero $x \in A_{ij}$ such that $xy=0$, then $y=0$.  We prove this in several steps.  Throughout the proof, we will let $N_{j} \subset e_{j}A$ denote the sum of all submodules of dimension $\leq 0$.
\newline
\newline
\noindent
{\it Step 1:  We prove that if $y \in A_{jl}$ has the property that there exists a nonzero $x \in A_{ij}$ such that $xy=0$, then either $x \in N_{i}$ or $y \in N_{j}$.}  To prove this, we let $K_{j}$ denote the graded submodule of $e_{j}A$ consisting of all $m$ such that $xm=0$.  Then $y \in K_{j}$, and if ${}_{x}\mu:e_{j}A \rightarrow e_{i}A$ denotes left multiplication by $x$, there is a short exact sequence
$$
0 \rightarrow K_{j} \rightarrow e_{j}A \overset{{}_{x}\mu}{\rightarrow} x e_{j}A \rightarrow 0.
$$
There are three possibilities for $\operatorname{dim }K_{j}$.  Either $\operatorname{dim }K_{j}=-1$, in which case $y = 0$ so that $y \in N_{j}$, or $\operatorname{dim }K_{j}=0$, in which case $K_{j} \subset N_{j}$ so that $y \in N_{j}$, or $\operatorname{dim }K_{j}=1$.  In the third case, $\operatorname{dim}_{K}(K_{j})_{l}$ will eventually have the form $cl+d$, where $c$ is an integer $\geq 1$.  Since $\operatorname{dim}_{K}(e_{j}A)_{l}$ is eventually of the form $l-j+1$, we must have $c=1$ so that $\operatorname{dim}(e_{j}A/K_{j}) \leq 0$.  It follows $\operatorname{dim}(x e_{j}A) \leq 0$ so that $x \in N_{i}$.
\newline
\newline
\noindent
{\it Step 2:  We prove that it suffices to show that $N_{i}=0$ for all $i$.}  For, if this is the case, then by Step 1, if $y \in A_{jl}$ has the property that there exists a nonzero $x \in A_{ij}$ such that $xy=0$, then $y=0$.
\newline
\newline
\noindent
{\it Step 3:  We prove that it suffices to show that for all $i$, there exists a homogeneous $z_{i} \in Ae_{i}$ such that $z_{i}N_{i}=0$ and such that $z_{i}$ is not an element of $\oplus_{j} N_{j}$.}  Suppose this is the case.  By Step 2, it suffices to prove that $N_{i}=0$ for all $i$.  To this end, suppose $z_{i} \in A_{li}$, let ${}_{z_{i}}\mu:e_{i}A \rightarrow e_{l}A$ denote left multiplication by $z_{i}$ and consider the exact sequence
\begin{equation} \label{eqn.firstsequence}
0 \rightarrow L \rightarrow e_{i}A \overset{{}_{z_{i}}\mu}{\rightarrow} e_{l}A \rightarrow e_{l}A/z_{i}e_{i}A \rightarrow 0.
\end{equation}
By the hypothesis on $z_{i}$, $N_{i} \subset L$.  On the other hand, since $z_{i}$ is not an element of $\oplus_{j} N_{j}$, $\operatorname{dim }z_{i}e_{i}A=1$.  Therefore, since there is a short exact sequence
$$
0 \rightarrow L \rightarrow e_{i}A \rightarrow z_{i}e_{i}A \rightarrow 0
$$
and since $\operatorname{dim}_{K}(e_{i}A)_{j}$ is eventually $j-i+1$, it follows that $\operatorname{dim }L \leq 0$.  Therefore, $L \subset N_{i}$ so that $L=N_{i}$.  Either $z_{i} \in A_{ii}$ which implies that $N_{i}=0$, or (\ref{eqn.firstsequence}) satisfies the hypotheses of Corollary \ref{cor.global}, in which case it follows that $L=N_{i}$ is free and has dimension $\leq 0$.  Therefore, $N_{i}=0$ as desired.
\newline
\newline
\noindent
{\it Step 4:  We prove that $\oplus_{i}N_{i}$ is a graded left module with $i$th component $N_{i}$.}  It suffices to show that if $x \in A_{ij}$ and $n \in N_{j}$ then $xn \in N_{i}$.  This follows from the fact that left multiplication by $x$ induces a morphism of right modules $N_{j} \rightarrow e_{i}A$ and the dimension of the image of a homomorphism is less than or equal to the dimension of the domain.
\newline
\newline
\noindent
{\it Step 5:  We prove that if $n \in N_{i}$ is homogeneous, the graded left module generated by $n$ has left dimension $\leq 0$.  Therefore, the left dimension of $\oplus_{i}N_{i}$ is $\leq 0$.}  The second statement follows from the first since left dimension is a dimension function on the category of graded left $A$-modules.  We now prove the first statement.  To this end, we begin with the observation that since $e_{i}A \cong e_{i+2j}A[2j]$ for any $j \in \mathbb{Z}$ by the proof of \cite[Proposition 13.2]{chan}, it follows that $N_{i} \cong N_{i+2j}[2j]$.  Next, we note that since $e_{i}A$ is noetherian, there exists integers $d_{i}, l_{i}$ such that $\operatorname{dim }_{K}(N_{i})_{l}=d_{i}$ for all $l>l_{i}$.  We claim that for all $j \in \mathbb{Z}$, $d_{i+2j}=d_{i}$ and $l_{i+2j}$ can be taken to be $l_{i}+2j$.  To prove the claim, we note that $\operatorname{dim}_{K}(N_{i+2j}[2j])_{l}=d_{i}$ for all $l>l_{i}$ by the initial observation, which implies that $\operatorname{dim}_{K}(N_{i+2j})_{l+2j}=d_{i}$ for all $l>l_{i}$ and therefore $\operatorname{dim}_{K}(N_{i+2j})_{l}=d_{i}$ for all $l>l_{i}+2j$.  The claim follows.

Now let $m \in N_{i}$ with $m \in A_{in}$.  Then $A_{i-2j,i}m$ is a left $K$-subspace of $(N_{i-2j})_{n}$, and by the claim, the latter space has dimension $d_{i}$ for all $n>l_{i}-2j$, i.e. for all $j>(l_{i}-n)/2$.  It follows that $Ae_{i}m$ has left dimension $\leq 0$, as desired.
\newline
\newline
\noindent
{\it Step 6:  We prove that for all $i$, there exists a homogeneous $z_{i} \in Ae_{i}$ such that $z_{i}N_{i}=0$ and such that $z_{i}$ is not an element of $\oplus_{j} N_{j}$.  The theorem will then follow from Step 3.}  Let $I \subset Ae_{i}$ denote the graded left ideal of all $x$ such that $xN_{i}=0$, let $n_{1},\ldots,n_{l}$ denote a set of homogeneous generators of $N_{i}$ as a graded right $A$-module, and let $I_{\nu} \subset Ae_{i}$ denote the graded left ideal of all $x$ such that $xn_{\nu}=0$.  Then $I= \cap I_{\nu}$ and $Ae_{i}n_{\nu} \cong Ae_{i}/I_{\nu}$ so that
\begin{eqnarray*}
\operatorname{dim }(Ae_{i}/I) & \leq & \operatorname{max}\{\operatorname{dim }(Ae_{i}/I_{\nu})\} \\
& = & \operatorname{max }\{\operatorname{dim }(Ae_{i}n_{\nu})\} \\
& \leq & \operatorname{dim }(\oplus_{j}N_{j}) \\
& \leq & 0
\end{eqnarray*}
where the second to last inequality follows from Step 4 and the last inequality follows from Step 5.  Hence, $\operatorname{dim }I=1$, which establishes the result.
\end{proof}
The next result now follows easily from Theorem \ref{thm.domain} in light of the definition of $C$.
\begin{cor} \label{cor.domain}
The algebra $C$ is a domain.
\end{cor}
Since $A$ is right and left noetherian, so is $C$.  Since $C$ is a domain, the set of nonzero homogeneous elements forms both a left and a right Ore set, and
$$
Q := \mbox{Frac}_{gr}C
$$
is both a left and a right ring of fractions.

The following result invokes the notion of an integral noncommutative space and big injective defined and studied in \cite{Smith}.  In order to understand the notation employed in the result and in the following section, the reader may wish to review the comments on notation and conventions concluding the introduction.

\begin{cor} \label{cor.integral}
The space ${\sf Proj }C$ is integral with big injective $\pi_{C} Q$.
\end{cor}

\begin{proof}
Since $C$ is a domain by Corollary \ref{cor.domain} which is generated in degree 1, all hypotheses of \cite[Theorem 4.5]{Smith} hold for $C$ except that $C_{0} = K$ acts centrally on $C$.  However, the proof still works in our context.
\end{proof}

\subsection{Grothendieck's Theorem}
The goal of this section is to classify vector bundles (defined before Corollary \ref{cor.grothendieck}) over $\mathbb{P}(V)$.  We will need a number of preliminary results.
\begin{lemma} \label{lemma.sums}
Let $l \in \mathbb{N}$.  If $x_{1},\ldots, x_{n}$ are nonzero homogeneous elements in $Q^{\oplus l}$ then the graded right $C$-module generated by $x_{1},\ldots, x_{n}$ is contained in a submodule isomorphic to a finite direct sum of shifts of $C$.
\end{lemma}

\begin{proof}
If $x_{1},\ldots, x_{n}$ are independent over $C$ then the result follows.  Otherwise, after relabeling, suppose $x_{1},\ldots,x_{m}$, with $m<n$, is a maximal subset of $\{x_{1},\ldots,x_{n}\}$ which is independent over $C$.  Then there exists $c_{1},\ldots, c_{m} \in C$ homogeneous and not all zero, and $c_{m+1}$ homogeneous and nonzero, such that
$$
x_{1}c_{1}+\cdots+x_{m}c_{m} = x_{m+1}c_{m+1}.
$$
Thus, in $Q^{\oplus l}$, $x_{1}c_{1}c_{m+1}^{-1}+\cdots+x_{m}c_{m}c_{m+1}^{-1}=x_{m+1}$.  By the left Ore condition, for each $1 \leq j \leq m$, there exist nonzero $d_{j}, e_{j} \in C$ such that $c_{j}c_{m+1}^{-1}=e_{j}^{-1}d_{j}$.  Therefore,
$$
x_{1}e_{1}^{-1}d_{1}+\cdots+x_{m}e_{m}^{-1}d_{m}=x_{m+1}.
$$
From this and the right Ore condition, $\{y_{1}:=x_{1}e_{1}^{-1},\ldots,y_{m}:=x_{m}e_{m}^{-1}\}$ is independent over $C$ and generates a graded right $C$-module which contains $x_{1},\ldots, x_{m+1}$.

If $n=m+1$, we are done.  Otherwise, $\{y_{1},\ldots, y_{m},x_{m+2}\}$ is dependent over $C$ since $\{x_{1},\dots, x_{m}, x_{m+2}\}$ is, so we can repeat the argument above to find $z_{1},\ldots,z_{m} \in Q^{\oplus m}$ independent over $C$ generating a graded right $C$-module which contains $x_{1},\ldots,x_{m+2}$.  A repetition of this argument a finite number of times establishes the result.
\end{proof}

\begin{lemma} \label{lemma.pretwosidediso}
For all $i \in \mathbb{Z}$, the unit map $\eta_{i}:e_{i}A \rightarrow \omega \pi e_{i}A$ is an isomorphism.
\end{lemma}

\begin{proof}
By \cite[Theorem 4.11, and Corollary 4.12]{duality}, $\tau e_{i}A =R^{1}\tau e_{i}A =0$ and for any $M \in {\sf Gr }\mathbb{S}(V)$, there is an exact sequence
\begin{equation} \label{eqn.fourterm}
0 \rightarrow \tau M \rightarrow M \rightarrow \omega \pi M
\rightarrow \operatorname{R}^{1}\tau M \rightarrow 0
\end{equation}
whose center arrow is the unit map.  Therefore, applying (\ref{eqn.fourterm}) in the case $M=e_{i}A$, the third arrow in (\ref{eqn.fourterm}) is an isomorphism of right $A$-modules and the result follows.
\end{proof}

\begin{corollary} \label{cor.twosidediso}
If $\pi e_{i}A \cong \pi e_{j}A$, then $i=j$.
\end{corollary}

\begin{proof}
On the one hand, by Lemma \ref{lemma.pretwosidediso},
\begin{eqnarray*}
\operatorname{Hom}_{\mathbb{P}(V)}(\pi e_{i}A,\pi e_{j}A) & \cong & \operatorname{Hom}_{{\sf Gr }A}(e_{i}A, e_{j}A) \\
& \cong & A_{ji}.
\end{eqnarray*}
On the other hand, under the hypothesis,
$$
\operatorname{Hom}_{\mathbb{P}(V)}(\pi e_{i}A, \pi e_{j}A) \cong \operatorname{Hom}_{\mathbb{P}(V)}(\pi e_{i}A,\pi e_{i}A),
$$
so that, again by Lemma \ref{lemma.pretwosidediso}, $\operatorname{Hom}_{\mathbb{P}(V)}(\pi e_{i}A,\pi e_{j}A) \cong K$.  It follows from \cite[Theorem 6.1.2]{vandenbergh} that $j=i$.
\end{proof}

\begin{prop} \label{prop.cproperties}
There is an equivalence ${\sf Proj }C \rightarrow \mathbb{P}(V)$ sending $\pi_{C}C[i]$ to $\pi_{A} e_{-2i}A$.
\end{prop}

\begin{proof}
If we let $C'$ be the $2\mathbb{Z}$-graded ring defined by letting $C'_{2i} := C_{i}$, and by letting multiplication be that induced by multiplication in $C$, then there is a canonical equivalence ${\sf Gr }C \rightarrow {\sf Gr }C'$ sending $C[i]$ to $C'[2i]$ and inducing an equivalence ${\sf Proj }C \rightarrow {\sf Proj }C'$.  By the proof of \cite[Lemma 2.4]{quadrics}, if $B$ is the $2\mathbb{Z}$-Veronese of $A$, and if $\check{C'}$ denotes the $2 \mathbb{Z}$-algebra with $\check{C'}_{2i,2j}:=C'_{2j-2i}$ and with multiplication induced by multiplication in $C'$, then $B \cong \check{C'}$. Thus, by \cite[p. 380]{twistings}, there is a canonical equivalence ${\sf Gr }C' \rightarrow {\sf Gr }B$ sending $C'[2i]$ to $e_{-2i}B$ and inducing an equivalence ${\sf Proj }C' \rightarrow {\sf Proj }B$.  Finally, the functor $-\otimes_{B}A:{\sf Gr }B \rightarrow {\sf Gr }A$ defined in \cite[Section 2]{quadrics} sends $e_{j}B$ to $e_{j}A$ and induces an equivalence ${\sf Proj }B \rightarrow {\sf Proj }A$ by \cite[Lemma 2.5]{quadrics}. The proposition follows.
\end{proof}
For the rest of this section, we will identify the categories ${\sf Proj }C$ and $\mathbb{P}(V)$ via the equivalence in Proposition \ref{prop.cproperties} without further comment.

Following \cite{Smith}, we denote the big injective in $\mathbb{P}(V)$ (computed in Corollary \ref{cor.integral}) by $\mathcal{E}$.  Recall from \cite{Smith} that an object $\mathcal{M}$ in $\mathbb{P}(V)$ is {\it torsion} if $\operatorname{Hom}_{\mathbb{P}(V)}(\mathcal{M}, \mathcal{E})=0$, and is {\it torsion-free} if the only submodule of it that is torsion is $0$.  Let ${\sf T}$ denote the full subcategory of $\mathbb{P}(V)$ of torsion objects.  Then ${\sf T}$ is a localizing subcategory, and there is a map of noncommutative spaces $j: \mathbb{P}(V) / {\sf T} \rightarrow \mathbb{P}(V)$ \cite[Section 3]{Smith}.  Furthermore, if $\mathcal{M}$ is an object of $\mathbb{P}(V)$, then there is an exact sequence
\begin{equation}\label{eqn.torsion0}
0 \rightarrow \mathcal{T}_{1} \rightarrow \mathcal{M} \rightarrow j_{*}j^{*}\mathcal{M} \rightarrow \mathcal{T}_{2} \rightarrow 0
\end{equation}
with $\mathcal{T}_{1}, \mathcal{T}_{2} \in {\sf T}$.  In particular, if $\mathcal{M}$ is torsion-free, then there is a short exact sequence
\begin{equation} \label{eqn.torsion}
0 \rightarrow \mathcal{M} \rightarrow j_{*}j^{*}\mathcal{M} \rightarrow \mathcal{N} \rightarrow 0
\end{equation}
where $\mathcal{N} \in {\sf T}$.

\begin{theorem} \label{thm.grothendieck}
Every noetherian torsion-free object in $\mathbb{P}(V)$ is a finite sum of the form $\oplus \pi_{A} e_{i}A $, and
every noetherian object of $\mathbb{P}(V)$ is a direct sum of a torsion object and a torsion-free object.
\end{theorem}

\begin{proof}
We first claim that if $\mathcal{M}$ is torsion-free and noetherian, then $\mathcal{M}$ is a submodule of a finite direct sum of shifts of $\pi_{C}C$.  To this end, we apply $\omega_{C}$ to the short exact sequence (\ref{eqn.torsion}), to get an exact sequence
$$
0 \rightarrow \omega_{C}\mathcal{M} \rightarrow \omega_{C}j_{*}j^{*}\mathcal{M} \rightarrow \omega_{C}\mathcal{N}.
$$
By the proof of \cite[Proposition 3.9]{Smith}, $j_{*}j^{*}\mathcal{M}$ is a finite direct sum of copies of $\mathcal{E}$.  On the other hand, $\omega_{C}\mathcal{E}=\omega_{C} \pi_{C} Q$, and $Q \cong \omega_{C} \pi_{C}Q$ since $Q$ is injective and $\tau_{C}Q=0$.  Therefore, $\omega_{C} \mathcal{M}$ is a submodule of a finite direct sum of copies of $Q$ and the claim now follows from Lemma \ref{lemma.sums}.

Applying $\pi_{C}$ to this containment, we conclude that $\mathcal{M} \subset \pi_{C}(\oplus_{i} C[i])$.  By Proposition \ref{prop.cproperties}, $\mathcal{M} \subset \pi_{A}(\oplus_{i} e_{-2i}A)$.  We denote the cokernel of this inclusion by $\mathcal{C}$.  By Lemma \ref{lemma.pretwosidediso}, $\omega_{A}\mathcal{M} \subset \oplus_{i}e_{-2i}A$, with cokernel $D$ contained in $\omega_{A} \mathcal{C}$.  Furthermore, since $\tau_{A}(\omega_{A} \mathcal{C})=0$ and since $\tau_{A}$ is left exact, we deduce that $\tau_{A}D=0$.  It follows from Lemma \ref{lemma.homological} and \cite[Proposition 3.19]{duality} that $D$ has a minimal resolution with length at most $1$.  Therefore, since there is a short exact sequence
$$
0 \rightarrow \omega_{A}\mathcal{M}  \rightarrow \oplus_{i}e_{-2i}A \rightarrow D \rightarrow 0
$$
in ${\sf gr }A$, we conclude that $\omega_{A}\mathcal{M}$ is projective, hence free by Corollary \ref{cor.free}.

Finally, suppose that $\mathcal{M}$ is a noetherian object of $\mathbb{P}(V)$.  Since $j_{*}j^{*}\mathcal{M}$ is torsion-free, it follows from (\ref{eqn.torsion0}) that $\mathcal{M}$ is an extension of a torsion-free module by a torsion module.  Therefore, by the argument in the first two paragraphs, it suffices to prove that if $\mathcal{T}$ is a noetherian torsion object, then $\mbox{Ext}^{1}_{\mathbb{P}(V)}(\pi_{C} C[i],\mathcal{T})=0$.  To prove this, we first note that by Proposition \ref{prop.cproperties}, it suffices to show that $\mbox{Ext}^{1}_{\mathbb{P}(V)}(\pi_{A} e_{-2i}A,\mathcal{T})=0$.  However, by \cite[Proposition 13.2 and Theorem 13.3]{chan}, $\mbox{Ext}^{1}_{\mathbb{P}(V)}(\pi_{A} e_{-2i}A,\mathcal{T}) \cong \operatorname{Hom}_{\mathbb{P}(V)}(\mathcal{T},\pi_{A} e_{-2i+2}A)$, so it suffices to show that there is no nonzero homomorphism $\mathcal{T} \rightarrow \pi_{A} e_{-2i+2}A$.  However, $\pi_{A} e_{-2i+2}A \cong \pi_{C} C[i-1]$, so the result will follow if we can show this module is torsion-free.  To prove this, we note that, as in the proof of Corollary \ref{cor.integral}, $C[i-1]$ is isomorphic to a subobject of $Q$.  Therefore, by the exactness of $\pi_{C}$ and the fact that $\pi_{C}Q$ is torsion-free, we may conclude that $\pi_{C}C[i-1]$ is torsion-free as desired.
\end{proof}
The notion of {\it rank} of a module over an integral noncommutative space is defined in \cite[Definition 3.3]{Smith}.  The following is a straightforward consequence of \cite[Corollary 3.7]{Smith}.
\begin{lemma} \label{lemma.preserves}
If ${\sf Z}$ and ${\sf W}$ are integral locally noetherian spaces, and if $F:{\sf Z} \rightarrow {\sf W}$ is an equivalence, then $F$ sends torsion-free rank 1 objects to torsion-free rank 1 objects.
\end{lemma}

\begin{lemma} \label{lemma.rank1}
For each $i \in \mathbb{Z}$, the module $\pi_{A} e_{i}A$ is torsion-free of rank 1.
\end{lemma}

\begin{proof}
For the proof we let $\pi$ denote $\pi_{C}$.  We claim that $\pi C$ has rank 1.  Since the shift functor in ${\sf Gr }C$ induces an equivalence in ${\sf Proj }C$, it will follow from the claim and Lemma \ref{lemma.preserves} that $\pi C[i]$ has rank 1 as well.  To prove the claim, we must show that the left $\operatorname{Hom}_{{\sf Proj }C}(\pi Q, \pi Q)$-module $\operatorname{Hom}_{{\sf Proj }C}(\pi C, \pi Q)$ is simple.  Since $\operatorname{Hom}_{{\sf Proj }C}(\pi Q,  \pi Q)$ is a division ring, it thus suffices to show that the map $\operatorname{Hom}_{{\sf Proj }C}(\pi Q,  \pi Q) \rightarrow \operatorname{Hom}_{{\sf Proj }C}(\pi C,  \pi Q)$ induced by inclusion $C \rightarrow Q$ is surjective.  This follows from the fact that $\pi Q$ is an injective object in ${\sf Proj }C$.

To prove the lemma, we note that if $i$ is even, then $\pi_{A} e_{i}A \cong \pi C[-i/2]$ by Proposition \ref{prop.cproperties}, so that in this case, the result follows from the claim.  If $i$ is odd, then since $e_{i}A$ is a submodule of $e_{i-1}A$, we may conclude, by the exactness of $\pi_{A}$, that $\pi_{A} e_{i}A$ is a submodule of $\pi C[-i/2]$, and hence of $\pi Q$.  By \cite[Corollary 4.12]{duality}, $\pi_{A} e_{i}A \neq 0$, and the result follows.
\end{proof}
We define a {\it vector bundle over $\mathbb{P}(V)$} to be a noetherian torsion-free object, and we define a {\it line bundle over $\mathbb{P}(V)$} to be a vector bundle of rank 1.  We have the following classification of vector bundles over $\mathbb{P}(V)$.

\begin{cor} \label{cor.grothendieck}
Every vector bundle over $\mathbb{P}(V)$ is a finite direct sum of line bundles, the line bundles are the objects of the form $\pi e_{i}A$, and $\pi e_{i}A \cong \pi e_{j}A$ implies $i=j$.
\end{cor}

\begin{proof}
The first and second results follow from the first part of Theorem \ref{thm.grothendieck} and Lemma \ref{lemma.rank1}, and the last result is exactly Corollary \ref{cor.twosidediso}.
\end{proof}

\section{Canonical equivalences} \label{section.canonical}
In this section we define and study three canonical equivalences between noncommutative projective lines.  We will see, in Corollary \ref{cor.newfactor}, that any equivalence between noncommutative projective lines is a composition of these three.  In order to understand the notation employed throughout this section and the following two, the reader may wish to review the comments on notation and conventions concluding the introduction.  Throughout this section, we will let $P_{i} := e_{i}\mathbb{S}(V)$, $\mathcal{P}_{i} := \pi P_{i}$, $P'_{i}:=e_{i}\mathbb{S}(W)$ and $\mathcal{P}'_{i}:=\pi P'_{i}$.

\subsection{The functor $[i]$} \label{section.shift}
We first define the shift functor.  Suppose $i \in \mathbb{Z}$ and $A$ is a $\mathbb{Z}$-algebra.  For any right $A$-module $M$, we let $M[i]$ denote the right $A(i)$-module with $M[i]_{j}=M_{i+j}$ and with multiplication induced by that of $A$ on $M$.  We define shift on morphisms similarly.  By \cite[Lemma 3.1]{smithmaps}, shift by $i$ induces a functor on the level of ${\sf Proj }$, and we abuse notation by calling this functor $[i]$ as well.

We will also abuse notation repeatedly as follows: since there is a canonical isomorphism $\mathbb{S}(V)(i) \rightarrow \mathbb{S}(V^{i*})$ of $\mathbb{Z}$-algebras over $K$ (see \ref{eqn.canonicalz}), shifting by $[i]$ induces an equivalence ${\sf Gr} \mathbb{S}(V) \rightarrow {\sf Gr} \mathbb{S}(V^{i*})$ and an equivalence $\mathbb{P}(V) \rightarrow \mathbb{P}(V^{i*})$.  We call both of these equivalences $[i]$.  Similarly, shift by $-i$ induces equivalences in the opposite direction, which we call $[-i]$.

\begin{lemma} \label{lemma.shiftp}
Let $W = V^{i*}$.  Under the equivalence $[i]:\mathbb{P}(V) \rightarrow \mathbb{P}(W)$, the image of $\mathcal{P}_{l}$ is isomorphic to $\mathcal{P}'_{l-i}$.
\end{lemma}

\begin{proof}
By Theorem \ref{thm.grothendieck}, Lemma \ref{lemma.preserves} and Lemma \ref{lemma.rank1}, $\mathcal{P}_{l}[i] \cong \mathcal{P}'_{j}$ for some $j$.  Therefore, $\omega \pi (P_{l}[i]) \cong P'_{j}$ by Lemma \ref{lemma.pretwosidediso}.  Since $\tau P_{l}=0$, it follows that $\tau_{\mathbb{S}(W)}(P_{l}[i])=0$ so that there is an inclusion $P_{l}[i] \rightarrow P'_{j}$.  Comparing nonzero degrees, we conclude that $l-i \geq j$.  Using the functor $[-i]$ instead of $[i]$ allows us to reverse the argument and conclude that $l-i \leq j$ so that $j=l-i$, whence the result.
\end{proof}

\subsection{The functor $\Phi$}
Suppose $\phi:V \rightarrow W$ is an isomorphism of two-sided vector spaces.  Then it is straightforward to check that $\phi$ induces an isomorphism of $\mathbb{Z}$-algebras $\mathbb{S}(V) \rightarrow \mathbb{S}(W)$, which we also call $\phi$.

We denote by $\Phi$ the equivalence ${\sf Gr }\mathbb{S}(V) \rightarrow {\sf Gr }\mathbb{S}(W)$ defined as follows:  If $M$ is an object in ${\sf Gr }\mathbb{S}(V)$, we define $\Phi(M)_{i}:=M_{i}$ as a set, with $\mathbb{S}(W)$-module structure
$$
\Phi(M)_{i} \otimes  \mathbb{S}(W)_{ij} \overset{1\otimes \phi^{-1}}{\rightarrow} \Phi(M)_{i} \otimes  \mathbb{S}(V)_{ij} \overset{\mu}{\rightarrow} \Phi(M)_{j},
$$
where $\mu$ denotes the $\mathbb{S}(V)$-module multiplication on $M$.  If $f:M \rightarrow N$ is a morphism in ${\sf Gr }\mathbb{S}(V)$ and $m \in M_{i}$, we define $\Phi(f)_{i}(m)=f_{i}(m)$.  By \cite[Lemma 3.1]{smithmaps}, $\Phi$ descends to a functor $\mathbb{P}(V) \rightarrow \mathbb{P}(W)$.  We abuse notation by calling this $\Phi$ as well.  It is elementary to check that if $\phi_{1}:V \rightarrow W$ induces $\Phi_{1}:\mathbb{P}(V) \rightarrow \mathbb{P}(W)$ and $\phi_{2}:W \rightarrow U$ induces $\Phi_{2}:\mathbb{P}(W) \rightarrow \mathbb{P}(U)$ then $\Phi_{2} \Phi_{1}$ is naturally equivalent to the equivalence induced by $\phi_{2}\phi_{1}$.

\begin{lemma} \label{lemma.functorvalues1}
Suppose $\Phi:{\sf Gr }\mathbb{S}(V) \rightarrow {\sf Gr }\mathbb{S}(W)$ is induced by an isomorphism $\phi:V \rightarrow W$.
\begin{enumerate}
\item{} There is an isomorphism $\Phi P_{i} \rightarrow P'_{i}$, and thus $\Phi \mathcal{P}_{i} \cong \mathcal{P}'_{i}$.

\item{} Suppose $f:\Phi P_{i} \rightarrow P'_{i}$ is an isomorphism.  Then $f_{i}:K \rightarrow K$ equals left multiplication by a nonzero element $a \in K$, denoted ${}_{a}\mu$, and $f_{i+1}={}_{a}\mu (\phi^{i*})^{j}$ where $j=1$ if $i$ is even and $j=-1$ if $i$ is odd.
\end{enumerate}
\end{lemma}

\begin{proof}
 The proof of the first part of (1) is similar to the proof of Lemma \ref{lemma.shiftp}.  The details are left to the reader.  For the proof of the second part of (1), recall that by abuse of notation, the symbol $\Phi \mathcal{P}_{i}$ denotes $\pi \Phi \omega \mathcal{P}_{i}$.  Therefore, the second part of (1) follows from the first part of (1), as $\pi \Phi \omega \mathcal{P}_{i} \cong \pi \Phi P_{i}$, where in the last isomorphism we have invoked Lemma \ref{lemma.pretwosidediso}.


We now turn to the proof of (2).  Since $f$ is a graded $\mathbb{S}(W)$-module isomorphism, $f_{i}:\mathbb{S}(V)_{ii} \rightarrow \mathbb{S}(W)_{ii}$ is a right $K$-module homomorphism from $K$ to $K$.  Therefore, there exists a nonzero $a \in K$ such that $f_{i}={}_{a}\mu$.  Similarly, the fact that $f$ is compatible with multiplication by elements of $\mathbb{S}(W)_{i,i+1}$ implies that $f_{i+1}:V^{i*} \rightarrow W^{i*}$ has the indicated form.
\end{proof}

\begin{lemma} \label{lemma.bigphi}
Suppose, for $j=1,2$, that $\Phi_{j}:{\sf Gr }\mathbb{S}(V) \rightarrow {\sf Gr }\mathbb{S}(W)$ is induced by an isomorphism $\phi_{j}:V \rightarrow W$.
\begin{enumerate}
\item{}  Suppose $h: \Phi_{1} P_{i} \rightarrow \Phi_{2} P_{i}$ is an isomorphism.  There exists a nonzero $a \in K$ such that $h_{i}={}_{a}\mu$ where ${}_{a}\mu$ denotes left multiplication by $a$, and
$$
h_{i+1}= \begin{cases} {}_{a}\mu  (\phi_{2}^{i*})^{-1}  \phi_{1}^{i*} \mbox{ if $i$ is even, } \\ {}_{a}\mu   (\phi_{2}^{i*})  (\phi_{1}^{i*})^{-1} \mbox{ if $i$ is odd. }\end{cases}
$$
\item{}  Suppose $\eta:\Phi_{1} \rightarrow \Phi_{2}$ is an equivalence, where now $\Phi_{i}$ denotes the induced equivalence $\mathbb{P}(V) \rightarrow \mathbb{P}(W)$.  Then there exists nonzero $a, b \in K$ such that $\phi_{2}^{-1}\phi_{1}(v)= a \cdot v \cdot b$ for all $v \in V$.
\end{enumerate}
\end{lemma}

\begin{proof}
The first assertion of (1) follows from the fact that $(\Phi_{j} P_{i})_{i} = K$.  To prove the second part of (1), we note that by Lemma \ref{lemma.functorvalues1}(1), there exist isomorphisms $f_{1}: \Phi_{1}P_{i} \rightarrow P'_{i}$ and $f_{2}:\Phi_{2}P_{i} \rightarrow P'_{i}$.  Since $h$ is an isomorphism, and since every isomorphism $P'_{i} \rightarrow P'_{i}$ is left multiplication by a nonzero element of $K$, there exists a nonzero $b \in K$ such that the diagram
$$
\begin{CD}
(\Phi_{1}P_{i})_{i+1} & \overset{f_{1}}{\rightarrow} & (P'_{i})_{i+1} \\
@V{h_{i+1}}VV @VV{{}_{b}\mu}V \\
(\Phi_{2}P_{i})_{i+1} & \underset{f_{2}}{\rightarrow} & (P'_{i})_{i+1}
\end{CD}
$$
commutes.  The result now follows from Lemma \ref{lemma.functorvalues1}(2).

We now prove part (2).  We let $v \in V$ and we recall that the notation $\Phi_{i}\mathcal{P}_{j}$ means $\pi \Phi_{i} \omega \pi P_{j}$.  Consider the following commutative diagram
\begin{equation} \label{eqn.bimods}
\begin{CD}
\Phi_{1}\mathcal{P}_{1} & \rightarrow & \Phi_{1}\mathcal{P}_{0} \\
@V{\eta_{\mathcal{P}_{1}}}VV @VV{\eta_{\mathcal{P}_{0}}}V \\
\Phi_{2} \mathcal{P}_{1} & \rightarrow & \Phi_{2}\mathcal{P}_{0}
\end{CD}
\end{equation}
whose horizontals are induced by left multiplication by $v \in V$.  Applying $\omega$ to this diagram and noting that, by Lemma \ref{lemma.functorvalues1}(1), $\Phi_{j} P_{i} \cong P'_{i}$, it follows from Lemma \ref{lemma.pretwosidediso} and the naturality of the unit map that the diagram
$$
\begin{CD}
\Phi_{1} \omega \pi P_{1} & \rightarrow & \Phi_{1} \omega \pi P_{0} \\
@VVV @VVV \\
\Phi_{2} \omega \pi P_{1} & \rightarrow & \Phi_{2} \omega \pi P_{0}
\end{CD}
$$
whose horizontals are induced by left multiplication by $v$, and whose verticals are induced by $\omega \eta$ and the unit map, commutes.  Therefore, by Lemma \ref{lemma.pretwosidediso} and the naturality of the unit map, we deduce the existence of a commutative diagram
$$
\begin{CD}
(\Phi_{1} P_{1})_{1} & \rightarrow & (\Phi_{1} P_{0})_{1} \\
@VVV @VVV \\
(\Phi_{2} P_{1})_{1} & \rightarrow & (\Phi_{2} P_{0})_{1}
\end{CD}
$$
whose horizontals are induced by left multiplication by $v \in V$ and whose verticals are degree 1 components of isomorphisms.  The result now follows from the first part of the lemma.
\end{proof}

We now prove that the equivalence $\Phi$ is compatible with taking the Veronese.  We let $J \subset \mathbb{Z}$ denote the set of even integers and we let $B$ denote the $J$-Veronese of $\mathbb{S}(V)$.  We suppose $\Phi:{\sf Gr }\mathbb{S}(V) \rightarrow {\sf Gr }\mathbb{S}(V)$ is an equivalence induced by an isomorphism of two-sided vector spaces $\phi: V \rightarrow V$, and we let $\Phi_{B}:{\sf Gr }B \rightarrow {\sf Gr }B$ denote the induced equivalence.  If we let $\operatorname{Res }:{\sf Gr }\mathbb{S}(V) \rightarrow {\sf Gr }B$ denote the canonical restriction functor \cite[Section 2]{quadrics}, then it is easy to check that
\begin{equation} \label{eqn.rescommute}
\Phi_{B} \operatorname{Res} = \operatorname{Res} \Phi.
\end{equation}
Since the functors $\operatorname{Res}$, $\Phi$ and $\Phi_{B}$ preserve torsion objects, it follows from \cite[Lemma 3.1]{smithmaps} that they descend to the functors $\underline{\operatorname{Res}}:\mathbb{P}(V) \rightarrow {\sf Proj }B$, $\Phi:\mathbb{P}(V) \rightarrow \mathbb{P}(V)$ and ${\Phi}_{B}:{\sf Proj }B \rightarrow {\sf Proj }B$ having the property that
\begin{equation} \label{eqn.underres}
\pi_{B} \operatorname{Res} \cong \underline{\operatorname{Res}} \pi_{\mathbb{S}(V)},
\end{equation}
\begin{equation} \label{eqn.underdelta}
\pi_{\mathbb{S}(V)} \Phi \cong \Phi \pi_{\mathbb{S}(V)}
\end{equation}
and
\begin{equation} \label{eqn.underdeltab}
\pi_{B} \Phi_{B} \cong {\Phi}_{B} \pi_{B}.
\end{equation}
It is straightforward to check that ${\Phi}_{B}:{\sf Proj }B \rightarrow {\sf Proj }B$ is an equivalence. Furthermore, by \cite[Lemma 2.5]{quadrics}, $\underline{\operatorname{Res}}$ is an equivalence.

\begin{lemma} \label{lemma.rescompat}
Retain the notation above.  Then $\Phi_{B} \underline{\operatorname{Res}} \cong {\underline{\operatorname{Res}}}_{}\Phi$.
\end{lemma}

\begin{proof}
It suffices to prove
\begin{equation} \label{eqn.withpi}
{\Phi}_{B} \underline{\operatorname{Res}}\pi_{\mathbb{S}(V)} \cong {\underline{\operatorname{Res}}}_{} \Phi \pi_{\mathbb{S}(V)}
\end{equation}
since right composing both sides of (\ref{eqn.withpi}) by $\omega_{\mathbb{S}(V)}$ implies the result.  Hence we proceed to prove (\ref{eqn.withpi}).  We note that
\begin{eqnarray*}
{\Phi}_{B} \underline{\operatorname{Res}}\pi_{\mathbb{S}(V)} & \cong & {\Phi}_{B}  \pi_{B} \operatorname{Res} \\
& \cong & \pi_{B} \Phi_{B} \operatorname{Res} \\
& = & \pi_{B} \operatorname{Res} \Phi \\
& \cong & \underline{\operatorname{Res}} \pi_{\mathbb{S}(V)} \Phi \\
& \cong & \underline{\operatorname{Res}}_{} {\Phi} \pi_{\mathbb{S}(V)}
\end{eqnarray*}
where the first isomorphism is induced by (\ref{eqn.underres}), the second isomorphism is induced by (\ref{eqn.underdeltab}), the equality follows from (\ref{eqn.rescommute}), the fourth isomorphism is induced by (\ref{eqn.underres}), and the final isomorphism is induced by (\ref{eqn.underdelta}).
\end{proof}

\begin{lemma} \label{lemma.finalbigphi}
Suppose there exists nonzero $a, b \in K$ such that $\phi:V \rightarrow V$ is the bimodule map defined by $\phi(v)=a \cdot v \cdot b$.  Then the induced equivalence $\Phi:\mathbb{P}(V) \rightarrow \mathbb{P}(V)$ is naturally equivalent to the identity functor.
\end{lemma}

\begin{proof}
We first prove that if $a=b^{-1}$ then $\Phi:{\sf Gr }\mathbb{S}(V) \rightarrow {\sf Gr }\mathbb{S}(V)$ is naturally equivalent to the identity.  To prove this, one checks that right multiplication by $b^{-1}$, $\mu_{b^{-1}}:\Phi(M) \rightarrow M$, is a right module isomorphism natural in $M$.

Next, we suppose $a=1$.  In this case, by Lemma \ref{lemma.rescompat}, it suffices to show that the induced functor $\Phi_{B}:{\sf Gr }B \rightarrow {\sf Gr }B$ is naturally equivalent to the identity.  To this end, we note that if $\phi:V \rightarrow V$ equals right multiplication by $b$, $\mu_{b}$, then since $\mu_{b}^{*}$ is left multiplication by $b$, it follows that $\phi \otimes (\phi^{*})^{-1}: V \otimes V^{*} \rightarrow V \otimes V^{*}$ is the identity map.

The general case follows from the fact that if $\phi(v)=a \cdot v \cdot b$, then $\phi$ is a composition of $\phi_{1}:V \rightarrow V$ defined by $\phi_{1}(v)=v \cdot ab$ and $\phi_{2}:V \rightarrow V$ defined by $\phi_{2}(v)=a \cdot v \cdot a^{-1}$.
\end{proof}

\subsection{The functors $T_{\sigma}$ and $T_{\delta,\epsilon}$} \label{section.t}
Suppose that for each $i \in \mathbb{Z}$, $\sigma_{i} \in \operatorname{Gal }(K/k)$, and let $\sigma := \{\sigma_{i}\}_{i \in \mathbb{Z}}$.  Let $A$ denote a $\mathbb{Z}$-algebra, and let $A_{\sigma}$ denote the $\mathbb{Z}$-algebra with
$$
A_{\sigma, ij} := K_{\sigma_{i}^{-1}} \otimes  A_{ij} \otimes  K_{\sigma_{j}}
$$
and with multiplication induced by that of $A$.  By \cite[Section 4.2]{vandenbergh}, the functor $T_{\sigma}:{\sf Gr }A \rightarrow {\sf Gr }A_{\sigma}$ defined on objects by $T_{\sigma}(M)_{i}:=M \otimes K_{\sigma_{i}}$ with the natural multiplication induced by that of $A$, and on morphisms in the obvious way is an equivalence.  By \cite[Lemma 3.1]{smithmaps}, $T_{\sigma}$ descends to an equivalence ${\sf Proj }A \rightarrow {\sf Proj }A_{\sigma}$, and we abuse notation by calling this $T_{\sigma}$.

The following result is utilized in the proof of Theorem \ref{thm.twosidedisom}.
\begin{prop} \label{prop.secondperiodic}
Suppose $\operatorname{char }k \neq 2$, and for each $i \in \mathbb{Z}$, let $\tau_{i} \in \operatorname{Gal }(K/k)$, and let $\tau := \{\tau_{i}\}_{i \in \mathbb{Z}}$.  If there is an isomorphism $f: \mathbb{S}(V)_{\tau} \rightarrow \mathbb{S}(W)$ over $K$, then  $\tau_{i}=\tau_{0}$ for $i$ even and $\tau_{i}=\tau_{1}$ for $i$ odd.
\end{prop}

\begin{proof}
Since $f_{i,i+1}$ and $f_{i+1, i+2}$ are isomorphisms of two-sided vector spaces,
\begin{eqnarray*}
K_{\tau_{i+1}^{-1}} \otimes V^{i+1*} \otimes K_{\tau_{i+2}} & \cong & W^{i+1*} \\
& \cong & (W^{i*})^{*} \\
& \cong & (K_{\tau_{i}^{-1}} \otimes V^{i*} \otimes K_{\tau_{i+1}})^{*}
\end{eqnarray*}
which implies that
\begin{equation} \label{eqn.istar}
V^{i+1*} \cong V^{i+1*} \otimes K_{\tau_{i}\tau_{i+2}^{-1}}
\end{equation}
According to Lemma \ref{lemma.twosidedclass}, there are three possibilities for the structure of $V$.

First, we suppose that there exists $\sigma \in \operatorname{Gal}(K/k)$ such that $V \cong K_{\sigma} \oplus K_{\sigma}$.  By (\ref{eqn.istar}), if $i$ is even, then $K_{\sigma^{-1}} \oplus K_{\sigma^{-1}} \cong K_{\sigma^{-1}\tau_{i}\tau_{i+2}^{-1}} \oplus K_{\sigma^{-1}\tau_{i}\tau_{i+2}^{-1}}$.  It follows that $\tau_{i}=\tau_{i+2}$ for all even $i$.  If $i$ is odd, the result holds by a similar argument.

Next, suppose that there exists $\sigma, \tau \in \operatorname{Gal}(K/k)$ with $\sigma \neq \tau$ such that $V \cong K_{\sigma} \oplus K_{\tau}$.  By (\ref{eqn.istar}), if $i$ is even, we have that
$$
K_{\sigma^{-1}} \oplus K_{\tau^{-1}} \cong K_{\sigma^{-1}\tau_{i}\tau_{i+2}^{-1}} \oplus K_{\tau^{-1}\tau_{i}\tau_{i+2}^{-1}}.
$$
Therefore, either $\tau_{i}=\tau_{i+2}$ for all even $i$, or there exists an even $i$ such that $\tau^{-1}=\sigma^{-1}\tau_{i}\tau_{i+2}^{-1}$ and $\sigma^{-1}=\tau^{-1}\tau_{i}\tau_{i+2}^{-1}$.  In the latter case, it follows that $\sigma \tau^{-1}$ has order 2.  Since $f_{01}$ is an isomorphism, we deduce that $W \cong K_{\sigma'} \oplus K_{\tau'}$ with $\sigma' \neq \tau'$, and since $f_{i,i+2}$ is an isomorphism, there is an isomorphism
\begin{equation} \label{eqn.automorphs}
K_{\tau_{i}^{-1}} \otimes (K \oplus K_{\sigma \tau^{-1}} \oplus K_{\tau \sigma^{-1}}) \otimes K_{\tau_{i+2}} \cong K \oplus K_{\sigma' {\tau'}^{-1}} \oplus K_{\tau' {\sigma'}^{-1}}.
\end{equation}
The right-hand side of (\ref{eqn.automorphs}) has exactly one summand isomorphic to $K$.  On the other hand, since $\sigma \tau^{-1}$ has order 2, the left-hand side of (\ref{eqn.automorphs}) has two isomorphic nontrivial factors.  It follows that if $i$ is even, $\tau_{i}=\tau_{i+2}$ and the result follows in this case.  If $i$ is odd, the result holds by a similar argument.

Finally, we prove the result when $V$ is simple.  We begin with the proof that $\tau_{2i} = \tau_{0}$ for all $i \in \mathbb{Z}$.  There are three cases to consider, according to the structure of $V \otimes V^{*}$ described in Proposition \ref{prop.structure}.  First, suppose $V \cong V(\lambda)$ is described by the first part of Proposition \ref{prop.structure}.  Then, for $i \in \mathbb{Z}$, ${\mathbb{S}(V)}_{2i,2i+2} \cong K \oplus K^{\oplus 2}_{\delta}$ where $\delta \in \operatorname{Gal}(K/k)$ has order 2.  It follows that
\begin{equation} \label{eqn.first}
K_{\tau_{2i}^{-1}} \otimes {\mathbb{S}(V)}_{2i,2i+2} \otimes K_{\tau_{2i+2}} \cong K_{{\tau_{2i}}^{-1} \tau_{2i+2}} \oplus {K_{{\tau_{2i}}^{-1} \delta \tau_{2i+2}}}^{\oplus 2}.
\end{equation}
On the other hand, by our hypothesis, $W$ must also be described by the first part of Proposition \ref{prop.structure}.  Thus,
\begin{equation} \label{eqn.second}
K_{{\tau_{2i}}^{-1}} \otimes {\mathbb{S}(V)}_{2i,2i+2} \otimes K_{\tau_{2i+2}} \cong K \oplus K_{\gamma}^{\oplus 2}
\end{equation}
where $\gamma \in \operatorname{Gal}(K/k)$ has order 2.  Applying $\operatorname{Hom}_{K \otimes_{k} K}(-,K)$ to (\ref{eqn.second}) yields $K$.  Therefore, the application of $\operatorname{Hom}_{K \otimes_{k} K}(-,K)$ to (\ref{eqn.first}) must yield $K$ as well, and this is only possible if $\tau_{2i}=\tau_{2i+2}$ for all $i$.

Next, suppose that $V$ is described by the second or third part of Proposition \ref{prop.structure}, and suppose that $i \in \mathbb{Z}$ is even.  We deduce from (\ref{eqn.istar}) that $V^{*} \cong V^{*} \otimes K_{\tau_{i}{\tau_{i+2}}^{-1}}$.  Thus, by Lemma \ref{lemma.autocomp}, $\mu$ has the same $G$-orbit as $\mu \tau_{i}{\tau_{i+2}}^{-1}$.  If $\tau_{i} \neq \tau_{i+2}$, then the image of $\mu$ equals the image of $\epsilon \mu$, where $\epsilon \in G$ and $\epsilon \mu \neq \mu$.  It follows that $\overline{\lambda} \epsilon \mu$ sends $K$ to $K$, contrary to the assumption on $V$.  Therefore, $\tau_{i}=\tau_{i+2}$ for $i$ even so that $\tau_{i}=\tau_{0}$.

Finally, we must prove that if $V$ is simple then $\tau_{2i+1}=\tau_{1}$ for all $i \in \mathbb{Z}$.  By \cite[Theorem 3.13]{hart}, $V^{*}$ is simple, and the result follows from considering the three possibilities for the structure of $V^{*}$ according to Proposition \ref{prop.structure} and reasoning as above.
\end{proof}

In what follows, we will need a special kind of twist.  If $\delta, \epsilon \in \operatorname{Gal}(K/k)$ and a sequence $\zeta$ is defined by
$$
\zeta_{i} = \begin{cases} \delta \mbox{ if $i$ is even} \\ \epsilon \mbox{ if $i$ is odd,} \end{cases}
$$
then we define $A_{\delta, \epsilon} := A_{\zeta}$.  We will need the following

\begin{lemma} \label{lemma.canonicaladj}
If $i \in \mathbb{Z}$ and $\delta, \epsilon \in \operatorname{Gal}(K/k)$, then there exist canonical adjunctions
$$
(-\otimes (K_{{\delta}^{-1}} \otimes V \otimes K_{\epsilon})^{i*},-\otimes (K_{{\delta}^{-1}} \otimes V \otimes K_{\epsilon})^{i+1*}, \eta_{i},\epsilon_{i})
$$
such that the associated noncommutative symmetric algebra $\mathbb{S}(K_{{\delta}^{-1}} \otimes V \otimes K_{\epsilon})$ is canonically isomorphic to $\mathbb{S}(V)_{\delta, \epsilon}$.
\end{lemma}

\begin{proof}
The proof is almost identical to the proof of \cite[Theorem 4.1]{izuru}, and so we omit the details.
\end{proof}

Let $\delta, \epsilon \in \operatorname{Gal}(K/k)$ and let $T_{\delta,\epsilon}:{\sf Gr }\mathbb{S}(V) \rightarrow {\sf Gr }S(K_{\delta^{-1}} \otimes V \otimes K_{\epsilon})$ denote the composition
\begin{equation} \label{eqn.twisttwo}
{\sf Gr }\mathbb{S}(V) \overset{T_{\zeta}}{\rightarrow} {\sf Gr }\mathbb{S}(V)_{\zeta} \rightarrow {\sf Gr }S(K_{\delta^{-1}} \otimes V \otimes K_{\epsilon})
\end{equation}
whose second composite is the equivalence induced by the isomorphism from Lemma \ref{lemma.canonicaladj}.

\begin{lemma} \label{lemma.functorvalues2}
Suppose $\delta, \epsilon \in \operatorname{Gal}(K/k)$ and let $W=K_{\delta}^{-1} \otimes V \otimes K_{\epsilon}$.
\begin{enumerate}
\item{} There is an isomorphism $T_{\delta, \epsilon}(P_{i}) \rightarrow P'_{i}$ and thus $T_{\delta,\epsilon}(\mathcal{P}_{i}) \cong \mathcal{P}'_{i}$.

\item{} Suppose $f:T_{\delta,\epsilon} P_{i} \rightarrow P'_{i}$ is an isomorphism.  Let $\zeta$ denote the sequence of automorphisms
$$
\zeta_{i} = \begin{cases} \delta \mbox{ if $i$ is even} \\ \epsilon \mbox{ if $i$ is odd.} \end{cases}
$$
There exists a nonzero $a \in K$ such that $f_{i}: \mathbb{S}(V)_{ii} \otimes K_{\zeta_{i}} \rightarrow \mathbb{S}(W)_{ii}$ is defined by $f_{i}(1 \otimes b)={}_{a}\mu \zeta_{i}^{-1}(b)$, where ${}_{a}\mu$ denotes left multiplication by $a$.
\end{enumerate}
\end{lemma}

\begin{proof}
The proof of (1) is similar to the proof of Lemma \ref{lemma.shiftp} and we leave the details to the reader.

For the proof of (2), we identify $\mathbb{S}(V)_{ii} \otimes K_{\zeta_{i}}$ with $K_{\zeta_{i}}$.  Under this identification, $f_{i}:K_{\zeta_{i}} \rightarrow K$ is a right $K$-module map so that $f_{i}(b)=f_{i}(1 \cdot \zeta_{i}^{-1}(b))=f_{i}(1) \zeta_{i}^{-1}(b)$.  Thus, the assertion holds if we set $a=f_{i}(1)$.
\end{proof}

The following lemma, which will be employed in the proof of Theorem \ref{thm.finalfinalfinal}, has a straightforward proof, which is omitted.  It describes two compatibilities between $\Phi$ and $T_{\delta,\epsilon}$.

\begin{lemma} \label{lemma.semidirect}
Suppose $\delta, \delta', \epsilon, \epsilon' \in \operatorname{Gal}(K/k)$,
$$
\phi:K_{\delta^{-1}} \otimes V \otimes K_{\epsilon} \rightarrow V
$$
is an isomorphism of two-sided vector spaces with induced equivalence $\Phi$, and
$$
\phi':K_{(\delta \delta')^{-1}} \otimes V \otimes K_{\epsilon \epsilon'} \rightarrow K_{\delta'^{-1}} \otimes K_{{\delta}^{-1}} \otimes V \otimes K_{\epsilon} \otimes K_{\epsilon'}
$$
is the canonical isomorphism with induced equivalence $\Phi'$.

Then the diagram
$$
\begin{CD}
\mathbb{P}(K_{\delta^{-1}} \otimes V \otimes K_{\epsilon}) & \overset{\Phi}{\rightarrow} & \mathbb{P}(V) \\
@V{T_{\delta',\epsilon'}}VV @VV{T_{\delta',\epsilon'}}V \\
\mathbb{P}(K_{\delta'^{-1}} \otimes (K_{\delta^{-1}} \otimes V \otimes K_{\epsilon}) \otimes K_{\epsilon'}) & \rightarrow & \mathbb{P}(K_{\delta'^{-1}} \otimes V \otimes K_{\epsilon'})
\end{CD}
$$
whose bottom horizontal is induced by $K_{\delta'^{-1}} \otimes \phi \otimes K_{\epsilon'}$, and the diagram
$$
\begin{CD}
\mathbb{P}(V) & \overset{T_{\delta,\epsilon}}{\rightarrow} & \mathbb{P}(K_{\delta^{-1}} \otimes V \otimes K_{\epsilon}) \\
@V{T_{\delta \delta', \epsilon \epsilon'}}VV @VV{T_{\delta', \epsilon'}}V \\
\mathbb{P}(K_{(\delta \delta')^{-1}} \otimes V \otimes K_{\epsilon \epsilon'}) & \underset{\Phi'}{\rightarrow} &  \mathbb{P}(K_{\delta'^{-1}} \otimes K_{{\delta}^{-1}} \otimes V \otimes K_{\epsilon} \otimes K_{\epsilon'})
\end{CD}
$$
commutes up to isomorphism.
\end{lemma}

\section{Classification of noncommutative projective lines} \label{section.isom}
Our goal in this section is to classify noncommutative projective lines up to $k$-linear equivalence.  We first introduce notation that will be employed in this section.  We let $\pi, \omega$ and $\tau$ denote the usual quotient, section, and torsion functors associated to $\mathbb{S}(V)$, we let $P_{i} := e_{i}\mathbb{S}(V)$, we let $\mathcal{P}_{i} := \pi P_{i}$, and we let $H$ denote the $\mathbb{Z}$-algebra with
$$
H_{ij}=\operatorname{Hom }_{{\sf Gr }\mathbb{S}(V)}(P_{j}, P_{i})
$$
and with multiplication induced by composition.  It is straightforward to check that the map
$$
f_{ij}: \mathbb{S}(V)_{ij} \rightarrow \operatorname{Hom }_{{\sf Gr
}\mathbb{S}(V)}(P_{j}, P_{i})
$$
defined by sending $x$ to the function sending $e_{j}$ to $x$
induces an isomorphism of $\mathbb{Z}$-algebras $f: \mathbb{S}(V)
\rightarrow H$.  We let $\mathcal{A}$ denote the $\mathbb{Z}$-algebra with
$$
\mathcal{A}_{ij}=\operatorname{Hom }_{{\sf Proj }{\mathbb{S}(V)}}(\mathcal{P}_{j}, \mathcal{P}_{i})
$$
and with multiplication induced by composition.  Finally, we let $g_{ij}:H_{ij} \rightarrow \mathcal{A}_{ij}$ be the map induced by the functor $\pi$.  Primed versions of this notation correspond to the analogous constructions with $W$ in place of $V$.

We give $H_{ij}$ a two-sided vector space structure through the
map $f$ and we give $\mathcal{A}_{ij}$ a two-sided vector space
structure through the map $gf$.  A similar remark holds for primed objects.

\begin{lemma} \label{lemma.twosidediso}
The map of $\mathbb{Z}$-algebras $g:H \rightarrow \mathcal{A}$ is a $K$-algebra isomorphism.  In particular, for any $i,j \in \mathbb{Z}$, $\mathbb{S}(V)_{ij}$ and $\mathcal{A}_{ij}$ are isomorphic two-sided vector spaces over $K$.
\end{lemma}

\begin{proof}
By the naturality of the unit map $\eta_{i}:P_{i} \rightarrow \omega \pi P_{i}$, the diagram
$$
\begin{CD}
\Hom_{{\sf Gr }\mathbb{S}(V)}(P_{j},P_{i}) & \rightarrow & \Hom_{{\sf Proj }\mathbb{S}(V)}(\mathcal{P}_{j}, \mathcal{P}_{i}) \\
@VVV @VVV \\
\Hom_{{\sf Gr }\mathbb{S}(V)}(P_{j},\omega \pi P_{i}) & \overset{=}{\rightarrow} & \Hom_{{\sf Gr }\mathbb{S}(V)}(P_{j},\omega \pi P_{i})
\end{CD}
$$
whose left vertical is induced by $\eta_{i}$, whose right vertical is induced by adjointness, and whose top horizontal is induced by $\pi$, commutes.  Since the verticals and bottom horizontal are isomorphisms by Lemma \ref{lemma.pretwosidediso}, so is the top horizontal, whence the result.
\end{proof}

\begin{thm} \label{thm.twosidedisom}
Suppose $F:\mathbb{P}(V) \rightarrow \mathbb{P}(W)$ is a $k$-linear equivalence.  Then there is an associated $i \in \mathbb{Z}$ and an associated sequence $\tau=\{\tau_{m}\}_{m \in \mathbb{Z}}$ in $\operatorname{Gal}(K/k)$, such that there is a $K$-algebra isomorphism
$$
h: \mathcal{A} \rightarrow \mathcal{A'}(i)_{\tau}.
$$
Furthermore, if $F$ is naturally equivalent to $G$ and $G$ has associated integer $j$ and sequence $\sigma$, then $j=i$ and $\sigma=\tau$.

If $\operatorname{char }k \neq 2$, then $\tau$ is 2-periodic.
\end{thm}

\begin{proof}
By Theorem \ref{thm.grothendieck}, Lemma \ref{lemma.preserves} and Lemma \ref{lemma.rank1}, for any $j \in \mathbb{Z}$, $F(\mathcal{P}_{j}) \cong \mathcal{P}'_{l}$ for some $l \in \mathbb{Z}$.  In particular, $F(\mathcal{P}_{0}) \cong \mathcal{P}'_{i}$ for some $i \in \mathbb{Z}$, and this $i$ is unique by Corollary \ref{cor.twosidediso}.  Furthermore, since $F$ induces a $k$-linear isomorphism
$$
\Hom_{{\sf Proj }\mathbb{S}(V)}(\mathcal{P}_{j},\mathcal{P}_{0}) \rightarrow \Hom_{{\sf Proj }\mathbb{S}(W)}(\mathcal{P}'_{l},\mathcal{P}'_{i}),
$$
we conclude, by Lemma \ref{lemma.twosidediso}, that $l=i+j$.  For each $j \in \mathbb{Z}$, choose an isomorphism $\theta_{j}: F(\mathcal{P}_{j}) \rightarrow \mathcal{P}'_{i+j}$. Then, for all $l,j \in \mathbb{Z}$, the composition $t_{lj}$
defined by:
\begin{eqnarray*}
\mathcal{A}_{lj}=\operatorname{Hom }_{{\sf Proj }\mathbb{S}(V)}(\mathcal{P}_{j}, \mathcal{P}_{l}) & \rightarrow & \operatorname{Hom }_{{\sf
Proj }\mathbb{S}(W)}(F \mathcal{P}_{j}, F \mathcal{P}_{l}) \\
& \rightarrow & \operatorname{Hom }_{{\sf Proj }\mathbb{S}(W)}(\mathcal{P}'_{i+j}, \mathcal{P}'_{i+l}) \\
& = & \mathcal{A'}_{i+l,i+j}
\end{eqnarray*}
whose first arrow is induced by $F$ and whose second arrow is
induced by the isomorphisms $\theta_{j}$ and $\theta_{l}$, is a
$k$-linear isomorphism.

For each $j \in \mathbb{Z}$ there exists a unique bijection $\tau_{j}$
making the diagram
$$
\begin{CD}
\mathcal{A}_{jj} &  &  & \overset{t_{jj}}{\longrightarrow} &
& &
\mathcal{A'}_{i+j,i+j} \\
@A{g_{jj} f_{jj}}AA &  & & & @AA{g_{i+j,i+j}' f_{i+j, i+j}'}A \\
K & & & \overset{\tau_{j}}{\longrightarrow} & & & K
\end{CD}
$$
commute.  Since $f$, $g$, $f'$, $g'$ and $t$ are $k$-linear
isomorphisms which respect multiplication, the bottom horizontal
is a $k$-linear isomorphism of $K$ onto $K$ which respects
multiplication, i.e. $\tau_{j} \in \operatorname{Gal}(K/k)$.

If $a \in K$ and $x
\in \mathcal{A}_{lj}$, we have
\begin{eqnarray*}
t_{lj}(a \cdot x) & = & t_{lj}(g_{ll} f_{ll} (a)
 x) \\
& = & \theta_{l} F(g_{ll} f_{ll}(a)  x)
\theta_{j}^{-1} \\
& = & \theta_{l} F(g_{ll} f_{ll}(a))  F(x)
\theta_{j}^{-1} \\
& = & \theta_{l} F(g_{ll} f_{ll}(a)) \theta_{l}^{-1}
\theta_{l} F(x) \theta_{j}^{-1} \\
& = & g_{i+l,i+l}' f_{i+l,i+l}' (\tau_{l}(a))  t_{lj}(x) \\
& = & \tau_{l}(a) \cdot t_{lj}(x).
\end{eqnarray*}
A similar computation establishes the fact that $t_{lj}(x
\cdot a)=t_{lj}(x) \cdot \tau_{j}(a)$.

We claim the map
$$
h_{lj}:\mathcal{A}_{lj} \rightarrow K_{\tau_{l}^{-1}} \otimes
\mathcal{A'}_{i+l,i+j} \otimes K_{\tau_{j}}
$$
defined by $h_{lj}(x)=1 \otimes t_{lj}(x) \otimes 1$ is
an isomorphism of two-sided vector spaces.  For, if $a \in K$,
\begin{eqnarray*}
h_{lj}(a \cdot x) & = & 1 \otimes t_{lj}(a \cdot x)
\otimes 1 \\
& = & 1 \otimes \tau_{l}(a) \cdot t_{lj}(x) \otimes 1 \\
& = & a \cdot (1 \otimes t_{lj}(x) \otimes 1) \\
& = & a \cdot h_{lj}(x).
\end{eqnarray*}
A similar computation shows that $h_{lj}$ is compatible with
right multiplication by $K$.  The fact that $h_{lj}$ is onto
and compatible with addition is clear, and the fact that
$h_{lj}$ is one-to-one follows by comparing $k$-dimension.  This establishes the claim.

Next we show that $h$ is a morphism of $\mathbb{Z}$-algebras.  In particular, we need to prove that $h$ is compatible with multiplication.  To this end, it suffices to prove that, for all $l,j,k \in \mathbb{Z}$, and for any $x \in \mathcal{A}_{lj}$ and $y \in \mathcal{A}_{jk}$,
$$
t_{lk}(x  y)=t_{lj}(x)  t_{jk}(y).
$$
We compute:
\begin{eqnarray*}
t_{lk}(xy) & = & \theta_{l}  F(xy)  \theta_{k}^{-1} \\
& = & \theta_{l}  F(x)  F(y)  \theta_{k}^{-1} \\
& = & \theta_{l} F(x) \theta_{j}^{-1}  \theta_{j}  F(y)  \theta_{k}^{-1} \\
& = & t_{lj}(x)  t_{kj}(y).
\end{eqnarray*}

To complete the proof of the first part of the result, we must show that if $G:\mathbb{P}(V) \rightarrow \mathbb{P}(W)$ is another $k$-linear equivalence with associated sequence of automorphisms $\sigma$ and associated shift by $j$, then $F$ naturally equivalent to $G$ implies $j=i$ and $\sigma=\tau$.  The proof will show that the automorphisms $\tau_{i}$ in the statement of the theorem are independent of choice of $\theta_{i}$.  The fact that $j=i$ follows from Corollary \ref{cor.twosidediso}.

To prove that $\sigma=\tau$, we must show that the diagram
\begin{equation} \label{eqn.tau}
\begin{CD}
\Hom_{\mathbb{P}(V)}(\mathcal{P}_{j}, \mathcal{P}_{j}) & \overset{F}{\rightarrow} &
\Hom_{\mathbb{P}(W)}(F\mathcal{P}_{j}, F\mathcal{P}_{j}) \\
@V{G}VV @VVV \\
\Hom_{\mathbb{P}(W)}(G\mathcal{P}_{j}, G\mathcal{P}_{j}) & \rightarrow &
\Hom_{\mathbb{P}(W)}(\mathcal{P}'_{j+i}, \mathcal{P}'_{j+i})
\end{CD}
\end{equation}
whose right vertical is induced by an isomorphism $\theta_{j}:F \mathcal{P}_{j} \rightarrow \mathcal{P}'_{j+i}$ and whose bottom horizontal is induced by an isomorphism $\gamma_{j}:G \mathcal{P}_{j} \rightarrow \mathcal{P}'_{j+i}$, commutes.

It is elementary to check that the diagram
$$
\begin{CD}
\Hom_{\mathbb{P}(V)}(\mathcal{P}_{j}, \mathcal{P}_{j}) & \overset{F}{\rightarrow} & \Hom_{\mathbb{P}(W)}(F\mathcal{P}_{j}, F\mathcal{P}_{j}) \\
@V{=}VV @VVV \\
\Hom_{\mathbb{P}(V)}(\mathcal{P}_{j}, \mathcal{P}_{j}) & \underset{G}{\rightarrow} & \Hom_{\mathbb{P}(W)}(G\mathcal{P}_{j}, G\mathcal{P}_{j})
\end{CD}
$$
whose right vertical is induced by an equivalence $F \rightarrow G$, commutes.  Therefore, to show that (\ref{eqn.tau}) commutes, it suffices to show that the diagram
\begin{equation} \label{eqn.lasty}
\begin{CD}
\Hom_{\mathbb{P}(W)}(G\mathcal{P}_{j}, G\mathcal{P}_{j})  & \rightarrow & \Hom_{\mathbb{P}(W)}(F\mathcal{P}_{j}, F\mathcal{P}_{j}) \\
@VVV @VVV \\
\Hom_{\mathbb{P}(W)}(\mathcal{P}'_{j+i}, \mathcal{P}'_{j+i}) & \underset{=}{\rightarrow} & \Hom_{\mathbb{P}(W)}(\mathcal{P}'_{j+i}, \mathcal{P}'_{j+i})
\end{CD}
\end{equation}
whose left vertical is induced by $\gamma$, whose right vertical is induced by $\theta$ and whose top horizontal is induced by the equivalence $\eta: F \rightarrow G$, commutes.  To prove this, let $x$ be an element of the upper left corner of (\ref{eqn.lasty}).  Then, via the upper route of (\ref{eqn.lasty}), $x$ goes to $\theta_{j} \eta^{-1} x \eta \theta_{j}^{-1}$, while via the lower route, it goes to $\gamma_{j} x \gamma_{j}^{-1}$.  Therefore, to complete the proof, we must show that $x={\gamma_{j}}^{-1} \theta_{j} \eta^{-1} x \eta \theta_{j}^{-1} \gamma_{j}$.  But $\eta \theta_{j}^{-1} \gamma_{j} \in \Hom_{\mathbb{P}(W)}(G\mathcal{P}_{j}, G\mathcal{P}_{j})$, and $\Hom_{\mathbb{P}(W)}(G\mathcal{P}_{j}, G\mathcal{P}_{j}) \cong K$ by Lemma \ref{lemma.twosidediso}, which is abelian under composition.  The first part of the theorem follows.

We complete the proof of the theorem by checking that if $\operatorname{char }k \neq 2$, then $\tau$ is 2-periodic.  Let $\tau^{-1}$ denote the sequence $\{\tau_{i}^{-1}\}_{i \in \mathbb{Z}}$ and note that there is an isomorphism of $\mathbb{Z}$-algebras over $K$ given by the composition
$$
\mathbb{S}(V)_{\tau^{-1}} \rightarrow \mathcal{A}_{\tau^{-1}} \rightarrow \mathcal{A'}(i) \rightarrow \mathbb{S}(W)(i)
$$
whose first and third arrows are induced by the isomorphism $g f$ and $(g' f')^{-1}$ defined at the beginning of this section and whose second arrow is induced by the isomorphism $h$.  Therefore, the result follows from Proposition \ref{prop.secondperiodic} in light of the fact that $\mathbb{S}(W)(i) = \mathbb{S}(W^{i*})$.
\end{proof}

\begin{cor} \label{cor.equiv}
If $F:\mathbb{P}(V) \rightarrow \mathbb{P}(W)$ is a $k$-linear equivalence, then there exists $\delta, \epsilon \in \operatorname{Gal}(K/k)$ such that
$$
V \cong K_{\delta} \otimes W \otimes
K_{\epsilon},
$$
or
$$
V \cong K_{\delta} \otimes W^{*} \otimes
K_{\epsilon}.
$$
\end{cor}

\begin{proof}
Let $h:\mathcal{A} \rightarrow \mathcal{A'}(i)_{\tau}$ be the $\mathbb{Z}$-algebra isomorphism over $K$ constructed in Theorem \ref{thm.twosidedisom}.  It follows from Lemma \ref{lemma.twosidediso} that $h_{01}$ induces an isomorphism of two-sided vector spaces
$$
V \rightarrow K_{\tau_{0}^{-1}} \otimes
\mathcal{A'}_{i,i+1} \otimes K_{\tau_{1}}.
$$
Again by Lemma \ref{lemma.twosidediso}, $\mathcal{A'}_{i,i+1}$ is isomorphic to either $W$ or $W^{*}$ depending on whether $i$ is even or odd.
\end{proof}
The converse of Corollary \ref{cor.equiv} holds as well, thanks to results in \cite{izuru}.

\begin{prop} \label{prop.izuru} \cite{izuru}
If there exist $\delta, \epsilon \in \operatorname{Gal}(K/k)$ with the property that either
$$
V \cong K_{\delta} \otimes W \otimes
K_{\epsilon},
$$
or
$$
V \cong K_{\delta} \otimes W^{*} \otimes
K_{\epsilon},
$$
then there is a $k$-linear equivalence $\mathbb{P}(V) \rightarrow \mathbb{P}(W)$.
\end{prop}

\begin{proof}
The fact that there is a $k$-linear equivalence $\mathbb{P}(W) \rightarrow \mathbb{P}(W^{*})$ follows from \cite[Lemma 4.2]{izuru} and the fact that there is a $k$-linear equivalence $\mathbb{P}(W) \rightarrow \mathbb{P}(K_{\delta} \otimes W \otimes K_{\epsilon})$ follows from the comments after the proof of \cite[Theorem 4.2]{izuru}.
\end{proof}
Combining our results with Lemma \ref{lemma.twosidedclass}, we obtain the following result, which invokes notation defined in the introduction.
\begin{thm} \label{thm.arithclass}
Suppose $\operatorname{char }k \neq 2$ and suppose that $V_{i}$ is a rank 2 two-sided vector space for $i=1,2$.  Then there is a $k$-linear equivalence $\mathbb{P}(V_{1}) \rightarrow \mathbb{P}(V_{2})$ if and only if
\begin{enumerate}
\item{} there exists $\sigma_{i} \in
\operatorname{Gal}(K/k)$ such that $V_{i} \cong K_{\sigma_{i}} \oplus K_{\sigma_{i}}$.  In this case, $\mathbb{P}(V_{i})$ is equivalent to the commutative projective line over $K$.

\item{} There exists $\sigma_{i}, \tau_{i}
\in \operatorname{Gal}(K/k)$, with $\sigma_{i} \neq \tau_{i}$, $V_{i} \cong K_{\sigma_{i}} \oplus K_{\tau_{i}}$ and under the (right) action of ${Gal(K/k)}^{2}$ on itself defined by $(\sigma, \tau) \cdot (\delta, \epsilon):= (\delta^{-1} \sigma \epsilon,\delta^{-1} \tau \epsilon)$ the orbit of $(\sigma_{1},\tau_{1})$ contains an element of the set
$$
\{(\sigma_{2},\tau_{2}), (\sigma_{2}^{-1},\tau_{2}^{-1}), (\tau_{2},\sigma_{2}), (\tau_{2}^{-1}, \sigma_{2}^{-1})\}.
$$

\item{}  $V_{i} \cong V(\lambda_{i})$, and under the action of ${Gal(K/k)}^{2}$ on $\Lambda(K)$ defined by $\lambda^{G} \cdot (\delta, \epsilon):= (\overline{\delta}^{-1} \lambda \epsilon)^{G}$, the orbit of $\lambda_{1}^{G}$ contains either $\lambda_{2}^{G}$ or $\mu_{2} := (\overline{\lambda_{2}})^{-1}|_{K}$.
\end{enumerate}
\end{thm}

\begin{proof}
By Lemma \ref{lemma.twosidedclass}, there are three possibilities for $V_{1}$.  If there exists $\sigma_{1} \in \operatorname{Gal}(K/k)$ such that $V_{1} \cong K_{\sigma_{1}} \oplus K_{\sigma_{1}}$ or there exist $\sigma_{1}, \tau_{1}
\in \operatorname{Gal}(K/k)$, with $\sigma_{1} \neq \tau_{1}$ and $V_{1} \cong K_{\sigma_{1}} \oplus K_{\tau_{1}}$, the result follows from Corollary \ref{cor.equiv} for the forward direction and Proposition \ref{prop.izuru} for the backward direction, in light of the fact that for $\sigma, \tau \in \operatorname{Gal}(K/k)$, $K_{\sigma} \cong K_{\tau}$ if and only if $\sigma=\tau$.

Now suppose $V_{1} \cong V(\lambda_{1})$ for some embedding $\lambda_{1}:K \rightarrow \overline{K}$. If there is an equivalence $\mathbb{P}(V_{1}) \rightarrow \mathbb{P}(V_{2})$, then, by Corollary \ref{cor.equiv}, $V_{2}$ is simple so that $V_{2} \cong V(\lambda_{2})$ and there exist $\delta, \epsilon \in \operatorname{Gal}(K/k)$ such that either $K_{\delta^{-1}} \otimes V(\lambda_{1}) \otimes K_{\epsilon} \cong V(\lambda_{2})$ or $K_{\delta^{-1}} \otimes V(\lambda_{1}) \otimes K_{\epsilon} \cong V(\mu_{2})$, where in the second case we have used \cite[Theorem 3.13]{hart}.  It follows from Lemma \ref{lemma.autocomp} that either $\lambda_{2}^{G}=(\overline{\delta}^{-1} \lambda_{1} \epsilon)^{G}$ or $\mu_{2}^{G}=(\overline{\delta}^{-1} \lambda_{1} \epsilon)^{G}$.  The proof of the converse is similar but uses Proposition \ref{prop.izuru} and we omit the details.
\end{proof}

We will need the following sharpening of Theorem \ref{thm.twosidedisom} to prove Theorem \ref{thm.finalfactor}.
\begin{prop} \label{prop.newnewfactor}
Suppose $\operatorname{char }k \neq 2$.  The isomorphism $h$ from Theorem \ref{thm.twosidedisom} can be chosen so that the composition
$$
\psi: \mathbb{S}(K_{\tau_{0}} \otimes V \otimes K_{\tau_{1}^{-1}}) \rightarrow \mathbb{S}(V)_{\tau^{-1}} \rightarrow \mathbb{S}(W)(i) \overset{\gamma}{\rightarrow} \mathbb{S}(W^{i*})
$$
whose first map is the isomorphism constructed in Lemma \ref{lemma.canonicaladj}, whose second map is induced by $h$, and whose third arrow is the map (\ref{eqn.canonicalz}), is an isomorphism
$$
\mathbb{S}(K_{\tau_{0}} \otimes V \otimes K_{\tau_{1}^{-1}}) \rightarrow \mathbb{S}(W^{i*})
$$
induced by an isomorphism of two-sided vector spaces $K_{\tau_{0}} \otimes V \otimes K_{\tau_{1}^{-1}} \rightarrow W^{i*}$.
\end{prop}

\begin{proof}
The isomorphism $h$ from Theorem \ref{thm.twosidedisom}, and hence $\psi$, depends on the choice of isomorphisms $\theta_{j}:F\mathcal{P}_{j} \rightarrow \mathcal{P}'_{i+j}$ for all $j \in \mathbb{Z}$.  Given an initial choice of $\theta_{j}$'s, we describe a modification of these choices producing an isomorphism
$$
\psi':\mathbb{S}(K_{\tau_{0}} \otimes V \otimes K_{\tau_{1}^{-1}}) \rightarrow \mathbb{S}(W^{i*})
$$
induced by an isomorphism of two-sided vector spaces $K_{\tau_{0}} \otimes V \otimes K_{\tau_{1}^{-1}} \rightarrow W^{i*}$.

By definition, the map $\psi$ is an isomorphism of noncommutative symmetric algebras over $K$.  By Corollary \ref{cor.gen2}, $\psi_{12}=(\psi_{01}^{-1})^{*} \mu_{a},$ where $\mu_{a}$ denotes right multiplication by a nonzero $a \in K$.  On the one hand, if $x \in V^{*}$, then, identifying $K_{\tau_{1}} \otimes V^{*} \otimes K_{\tau_{0}^{-1}}$ with $(K_{\tau_{0}}\otimes V \otimes K_{\tau_{1}^{-1}})^{*}$, we have
\begin{equation} \label{eqn.psi12}
\psi_{12}(1 \otimes x \otimes 1)=(\psi_{01}^{-1})^{*}(1\otimes x \otimes 1\cdot a).
\end{equation}
On the other hand, if ${}_{x}\mu:P_{2} \rightarrow P_{1}$ denotes left multiplication by $x$, then
\begin{equation} \label{eqn.psi12second}
\psi_{12}(1 \otimes x \otimes 1)=\gamma_{12}(g'f')^{-1}(\theta_{1}  F(\pi({}_{x}\mu))  \theta_{2}^{-1}).
\end{equation}
We let $\theta_{2}'=\mu_{a}  \theta_{2}$.  Then $(\theta_{2}')^{-1}=\theta_{2}^{-1}  \mu_{a^{-1}}$, and so
\begin{eqnarray*}
\gamma_{12}(g' f')^{-1}(\theta_{1}  F(\pi({}_{x}\mu))  {\theta_{2}'}^{-1})  & = & \gamma_{12}(g' f')^{-1}(\theta_{1} F(\pi({}_{x}\mu))  \theta_{2}^{-1} \cdot a^{-1}) \\
& = & \gamma_{12} (g' f')^{-1}( \theta_{1}  F(\pi({}_{x}\mu))  \theta_{2}^{-1}) \cdot a^{-1} \\
& = & ((\psi_{01}^{-1})^{*}(1\otimes x \otimes 1 \cdot a)) \cdot a^{-1} \\
& = &  (\psi_{01}^{-1})^{*}(1 \otimes x \otimes 1)
\end{eqnarray*}
where the third equality follows from (\ref{eqn.psi12second}) and (\ref{eqn.psi12}).

Continuing inductively, we define, for $j > 1$, $\theta_{j}'$ such that
$$
\psi_{j,j+1}'  = \begin{cases} \psi_{01}^{j*} \mbox{ if $j$ is even} \\ (\psi_{01}^{-1})^{j*} \mbox{ if $i$ is odd.} \end{cases}
$$
A similar argument, employing Corollary \ref{cor.gen1} in lieu of Corollary \ref{cor.gen2}, allows us to define, for $j < 0$, $\theta_{j}'$ such that
$$
\psi_{j,j+1}'  = \begin{cases} \psi_{01}^{j*} \mbox{ if $j$ is even} \\ (\psi_{01}^{-1})^{j*} \mbox{ if $i$ is odd.} \end{cases}
$$
\end{proof}

\section{Classification of equivalences}
The goal of this section, realized in Theorem \ref{thm.finalmain}, is to classify equivalences between noncommutative projective lines.

Throughout this section, we will employ the following notation:  $F:\mathbb{P}(V) \rightarrow \mathbb{P}(W)$ will denote a $k$-linear equivalence, $\tau=\{\tau_{i}\}$ will denote the associated sequence of automorphisms from Theorem \ref{thm.twosidedisom}, and $\tau^{-1}$ will denote the sequence $\{\tau_{i}^{-1}\}$.  As in Section \ref{section.isom}, we will let $P_{i}:=e_{i}\mathbb{S}(V)$, $\mathcal{P}_{i}:=\pi P_{i}$, $P'_{i}:=e_{i}\mathbb{S}(W)$ and $\mathcal{P}'_{i}:=\pi P'_{i}$.

We let the functor
$$
\Gamma:\mathbb{P}(V) \rightarrow {\sf Gr }\mathcal{A}
$$
be defined on objects by $\Gamma(\mathcal{M}):= \oplus_{j \in \mathbb{Z}} \Hom_{\mathbb{P}(V)}(\mathcal{P}_{j},\mathcal{M})$, with graded module structure induced by composition.  We let
$$
\Psi:{\sf Gr }\mathbb{S}(V) \rightarrow {\sf Gr }\mathcal{A}
$$
denote the equivalence induced by the isomorphism $gf: \mathbb{S}(V) \rightarrow \mathcal{A}$, where $f:\mathbb{S}(V) \rightarrow H$ and $g:H \rightarrow \mathcal{A}$ are isomorphisms defined at the beginning of Section \ref{section.isom}.  We let $\Psi_{\tau}: {\sf Gr }\mathbb{S}(V)_{\tau} \rightarrow {\sf Gr }\mathcal{A}_{\tau}$ denote the induced equivalence, we let $\Psi^{-1}:{\sf Gr }\mathcal{A} \rightarrow {\sf Gr }\mathbb{S}(V)$
denote the equivalence induced by $(g  f)^{-1}:\mathcal{A} \rightarrow \mathbb{S}(V)$, and we let $\Psi^{-1}_{\tau}$ denote the equivalence induced by $\Psi^{-1}$.  We define
$$
\overline{\pi}:{\sf Gr }\mathcal{A} \rightarrow \mathbb{P}(V)
$$
to be the composition $\pi \Psi^{-1}$, and we define
$$
\Delta:{\sf Gr }\mathcal{A}_{\tau^{-1}} \rightarrow {\sf Gr }\mathcal{A'}(i)
$$
be the functor induced by the isomorphism $h$ defined in Theorem \ref{thm.twosidedisom}.  Finally, we let
$$
\Delta':{\sf Gr }\mathbb{S}(V)_{\tau^{-1}} \rightarrow {\sf Gr }\mathbb{S}(W)(i)
$$
be the functor induced by the composition
$$
\mathbb{S}(V)_{\tau^{-1}} \rightarrow \mathcal{A}_{\tau^{-1}} \rightarrow \mathcal{A'}(i) \rightarrow \mathbb{S}(W)(i)
$$
whose first and third arrows are induced by the isomorphism $gf$ and $(g' f')^{-1}$ defined at the beginning of Section \ref{section.isom} and whose second arrow is induced by the isomorphism $h$ constructed in Theorem \ref{thm.twosidedisom}.  We will abuse notation by letting $\Delta'$ also denote the induced equivalence ${\sf Proj }\mathbb{S}(V)_{\tau^{-1}} \rightarrow {\sf Proj }\mathbb{S}(W)(i)$.

We will sometimes abuse notation by using identical notation when replacing $V$ by $W$ and $\mathcal{A}$ by $\mathcal{A'}$ when there is no chance of confusion.

\begin{lemma} \label{lemma.dig1}
The diagram of functors
$$
\begin{CD}
{\sf Gr }\mathbb{S}(V) & \overset{\Psi}{\rightarrow} & {\sf Gr }\mathcal{A} \\
@A{\omega}AA @AA{\Gamma}A \\
\mathbb{P}(V) & \underset{=}{\rightarrow} & \mathbb{P}(V)
\end{CD}
$$
commutes up to natural equivalence.
\end{lemma}

\begin{proof}
We need to prove that there is an isomorphism of functors $\Gamma \rightarrow \Psi \omega$.  Let $\mathcal{M}$ be an object in $\mathbb{P}(V)$.  We define a function $\eta_{\mathcal{M}}:\Gamma(\mathcal{M}) \rightarrow \Psi \omega \mathcal{M}$ by the composition
\begin{eqnarray*}
\Gamma(\mathcal{M})_{i} & = & \Hom_{\mathbb{P}(V)}(\mathcal{P}_{i}, \mathcal{M}) \\
& \cong & \Hom_{{\sf Gr }{\mathbb{S}(V)}}(P_{i},\omega \mathcal{M}) \\
& \cong & (\omega \mathcal{M})_{i} \\
& = & (\Psi \omega \mathcal{M})_{i}
\end{eqnarray*}
whose second map is induced by adjointness of $\pi$ and $\omega$, and whose third map is evaluation at $1 \in {\mathbb{S}(V)}_{ii}$.  The proof that this map is natural in $\mathcal{M}$ is straightforward and omitted.  Thus, to show that $\eta$ is a natural isomorphism, it remains to check that for all $i,j \in \mathbb{Z}$, the diagram
\begin{equation} \label{eqn.vitaldig}
\begin{CD}
\Gamma(\mathcal{M})_{i} \otimes \mathcal{A}_{ij} & \rightarrow & \Gamma(\mathcal{M})_{j} \\
@V{\eta}VV @VV{\eta}V \\
\Psi(\omega \mathcal{M})_{i} \otimes \mathcal{A}_{ij} & \rightarrow & \Psi(\omega \mathcal{M})_{j}
\end{CD}
\end{equation}
whose horizontals are multiplication, commutes.  To this end, suppose $s \in \Gamma(\mathcal{M})_{i}$ and $x_{ij} \in \mathbb{S}(V)_{ij}$.  We compute the image of $s \otimes gf(x_{ij}) \in \Gamma(\mathcal{M})_{i} \otimes \mathcal{A}_{ij}$ via the two routes of (\ref{eqn.vitaldig}).  Via the upper route, we end up with the element
$$
\omega s  \omega \pi ({}_{x_{ij}}\mu)  \eta_{j}(1)
$$
where $\eta_{j}:P_{j} \rightarrow \omega \pi P_{j}$ is the unit map and ${}_{x_{ij}}\mu:P_{j} \rightarrow P_{i}$ denotes left multiplication by $x_{ij}$.  If we map $s \otimes gf(x_{ij})$ via the lower route, we obtain $(\omega s  \eta_{i}(1)) \cdot x_{ij}$, where $\eta_{i}$ is defined analogously to $\eta_{j}$.  Thus, since $\eta_{i}$ is compatible with right $\mathbb{S}(V)$-module multiplication, it suffices to prove that
$$
\omega \pi({}_{x_{ij}}\mu)  \eta_{j}(1) = \eta_{i}(x_{ij}).
$$
This equality follows from the naturality of the unit of an adjoint pair.
\end{proof}
It follows from Lemma \ref{lemma.dig1} that
\begin{equation} \label{eqn.pibar}
\overline{\pi}\Gamma \cong {\sf id}_{\mathbb{P}(V)}.
\end{equation}

We omit the routine verification of the following
\begin{lemma} \label{lemma.dig2}
The diagram of functors
$$
\begin{CD}
{\sf Gr }\mathbb{S}(V) & \overset{\Psi}{\rightarrow} & {\sf Gr }\mathcal{A} \\
@V{T_{\tau}}VV @VV{T_{\tau}}V \\
{\sf Gr }\mathbb{S}(V)_{\tau}  & \underset{\Psi_{\tau}}{\rightarrow} & {\sf Gr }\mathcal{A}_{\tau}
\end{CD}
$$
commutes exactly.
\end{lemma}

\begin{prop} \label{prop.factor}
The diagram
$$
\begin{CD}
{\sf Gr }\mathcal{A} & \overset{T_{\tau^{-1}}}{\rightarrow} & {\sf Gr }\mathcal{A}_{\tau^{-1}} & \overset{\Delta}{\rightarrow} & {\sf Gr }\mathcal{A'}(i) \\
@A{\Gamma}AA & & @AA{[i]}A \\
\mathbb{P}(V) & \underset{F}{\rightarrow} & \mathbb{P}(W) & \underset{\Gamma}{\rightarrow} & {\sf Gr }\mathcal{A'}
\end{CD}
$$
commutes up to natural equivalence of functors.
\end{prop}

\begin{proof}
Let $\mathcal{M}$ be an object in $\mathbb{P}(V)$.  We define an isomorphism
\begin{equation} \label{eqn.gamma}
\eta_{\mathcal{M}}:  \Delta T_{\tau^{-1}} \Gamma \mathcal{M} \rightarrow (\Gamma F \mathcal{M})[i]
\end{equation}
in ${\sf Gr }\mathcal{A'}(i)$, and we show that $\eta$ is natural in $\mathcal{M}$.

We begin by noting that $(\Delta T_{\tau^{-1}} \Gamma \mathcal{M})_{l}=\Hom_{\mathbb{P}(V)}(\mathcal{P}_{l},\mathcal{M}) \otimes  K_{\tau_{l}^{-1}}$.  On the other hand, $(\Gamma F \mathcal{M})[i]_{l} = \operatorname{Hom }_{\mathbb{P}(W)}(\mathcal{P}'_{i+l},F\mathcal{M})$, so that in order to define an isomorphism (\ref{eqn.gamma}), we must first define an abelian group isomorphism
$$
(\eta_{\mathcal{M}})_{l}:\operatorname{Hom}_{\mathbb{P}(V)}(\mathcal{P}_{l},\mathcal{M})\otimes K_{\tau_{l}^{-1}} \rightarrow \operatorname{Hom}_{{\mathbb{P}(W)}}(\mathcal{P}'_{i+l}, F\mathcal{M})
$$
and prove that this isomorphism is natural in $\mathcal{M}$ and respects $\mathcal{A'}(i)$-module multiplication.

We define $(\eta_{\mathcal{M}})_{l}$ as the composition
\begin{eqnarray*}
\Hom_{\mathbb{P}(V)}(\mathcal{P}_{l},\mathcal{M}) \otimes K_{\tau_{l}^{-1}} & \rightarrow & \Hom_{\mathbb{P}(V)}(\mathcal{P}_{l}, \mathcal{M}) \\
& \overset{F}{\rightarrow} & \Hom_{\mathbb{P}(W)}(F \mathcal{P}_{l},F\mathcal{M}) \\
& \rightarrow & \Hom_{\mathbb{P}(W)}(\mathcal{P}'_{i+l}, F \mathcal{M})
\end{eqnarray*}
whose first arrow is induced by the assignment $m \otimes a \mapsto m \cdot a$ and whose last arrow is induced by the map $\theta_{l}:F\mathcal{P}_{l} \rightarrow \mathcal{P}'_{i+l}$ chosen in the proof of Theorem \ref{thm.twosidedisom}.  The fact that $(\eta_{\mathcal{M}})_{l}$ is an abelian group isomorphism which is natural in $\mathcal{M}$ is straightforward and the verification is omitted.

We proceed to check that $\eta$ is compatible with $\mathcal{A'}(i)$-module multiplication.  Let $m \in \Hom_{\mathbb{P}(V)}(\mathcal{P}_{l},\mathcal{M})$ and let $y \in \mathcal{A'}_{i+l,i+j}$.  Let $x \in \mathcal{A}_{lj}$ be such that $y=t_{lj}(x)$, where $t:\mathcal{A} \rightarrow \mathcal{A'}(i)$ is the map defined in the proof of Theorem \ref{thm.twosidedisom}.  Then, as the reader can check, the proof of compatibility will follow from the fact that $F(m)F(x) \theta_{j}^{-1} = F(m)\theta_{l}^{-1} y$.  Therefore, it suffices to check that $\theta_{l}F(x)\theta_{j}^{-1}=t_{lj}(x)$.  This last equality follows from the definition of $t$ and the proof follows.
\end{proof}

\begin{cor} \label{cor.factor}
There is an isomorphism
$$
F \cong \overline{\pi} [-i] \Delta T_{\tau^{-1}} \Gamma.
$$
\end{cor}

\begin{proof}
By Proposition \ref{prop.factor}, there is an equivalence $\Gamma F \cong [-i] \Delta T_{\tau^{-1}}\Gamma$.  The result follows by applying $\overline{\pi}$ to the left-hand side of the equation and invoking (\ref{eqn.pibar}).
\end{proof}

\begin{prop} \label{prop.mainfactor}
There is an isomorphism $F \cong \pi [-i] \Delta' T_{\tau^{-1}} \omega_{\mathbb{S}(V)}$.
\end{prop}

\begin{proof}
We show that the diagram
$$
\begin{CD}
{\sf Gr }\mathbb{S}(V)_{\tau^{-1}} & \overset{\Psi_{\tau^{-1}}}{\rightarrow} & {\sf Gr }\mathcal{A}_{\tau^{-1}} & \overset{\Delta}{\rightarrow} & {\sf Gr }\mathcal{A'}(i) & \rightarrow & {\sf Gr }\mathbb{S}(W)(i) \\
@A{T_{\tau^{-1}}}AA @AA{T_{\tau^{-1}}}A @VV{[-i]}V @VV{[-i]}V \\
{\sf Gr }\mathbb{S}(V) & \underset{\Psi}{\rightarrow} & {\sf Gr }\mathcal{A} & & {\sf Gr }\mathcal{A'} & \overset{\Psi^{-1}}{\rightarrow} & {\sf Gr }\mathbb{S}(W) \\
@A{\omega}AA @AA{\Gamma}A @V{\overline{\pi}}VV @VV{\pi}V \\
\mathbb{P}(V) & \underset{=}{\rightarrow} & \mathbb{P}(V) & \underset{F}{\rightarrow} & \mathbb{P}(W) & \underset{=}{\rightarrow} & \mathbb{P}(W)
\end{CD}
$$
whose upper right arrow is induced by $\Psi^{-1}$, commutes up to natural equivalence.  The bottom left square commutes by Lemma \ref{lemma.dig1}, the top left square commutes by Lemma \ref{lemma.dig2}, the central rectangle commutes by Corollary \ref{cor.factor}, the upper right square commutes by definition, and the bottom right square commutes by the definition of $\overline{\pi}$.  The result follows from the fact that the top row equals the functor $\Delta'$.
\end{proof}

\begin{cor} \label{cor.newfactor}
The diagram
\begin{equation} \label{eqn.factor}
\begin{CD}
\mathbb{P}(V) & \overset{F}{\rightarrow} & \mathbb{P}(W) \\
@V{T_{\tau^{-1}}}VV @AA{[-i]}A \\
{\sf Proj }\mathbb{S}(V)_{\tau^{-1}} & \underset{\Delta'}{\rightarrow} & {\sf Proj }\mathbb{S}(W)(i)
\end{CD}
\end{equation}
commutes up to natural equivalence.
\end{cor}

\begin{proof}
The result follows from Proposition \ref{prop.mainfactor} and \cite[Lemma 3.1(1)]{smithmaps}.  We leave the details of the straightforward proof to the reader.
\end{proof}

The next result follows directly from Corollary \ref{cor.newfactor} and Proposition \ref{prop.newnewfactor}.
\begin{theorem} \label{thm.finalfactor}
Suppose $\operatorname{char }k \neq 2$.  There exists $\delta, \epsilon \in \operatorname{Gal}(K/k)$, $i \in \mathbb{Z}$ and an equivalence
$$
\Phi:\mathbb{P}(K_{\delta^{-1}} \otimes V \otimes K_{\epsilon}) \rightarrow \mathbb{P}(W^{i*})
$$
induced by an isomorphism of two-sided vector spaces $\phi: K_{\delta^{-1}} \otimes V \otimes K_{\epsilon} \rightarrow W^{i*}$ such that the diagram
\begin{equation} \label{eqn.finalfactor}
\begin{CD}
\mathbb{P}(V) & \overset{F}{\rightarrow} & \mathbb{P}(W) \\
@V{T_{\delta,\epsilon}}VV @AA{[-i]}A \\
\mathbb{P}(K_{\delta^{-1}} \otimes V \otimes K_{\epsilon}) & \underset{\Phi}{\rightarrow} & \mathbb{P}(W^{i*})
\end{CD}
\end{equation}
commutes up to natural equivalence.
\end{theorem}
We next prove that the integer $i$ and the automorphisms $\delta, \epsilon \in \operatorname{Gal}(K/k)$ from Theorem \ref{thm.finalfactor} are unique up to natural equivalence, and we explore the dependence of the equivalence $\Phi$ on $F$ from Theorem \ref{thm.finalfactor}.

\begin{prop} \label{prop.unique}
Retain the notation and hypotheses from Theorem \ref{thm.finalfactor}.  The integer $i$, $\delta$, and $\epsilon$ are unique up to natural equivalence.
\end{prop}

\begin{proof}
First, suppose $[-i] \Phi_{1} T_{\delta_{1}, \epsilon_{1}} \cong [-j] \Phi_{2} T_{\delta_{2}, \epsilon_{2}}$.  By Lemma \ref{lemma.functorvalues1}(1), Lemma \ref{lemma.functorvalues2}(1) and Lemma \ref{lemma.shiftp}, the left-hand side takes $\mathcal{P}_{l}$ to $\mathcal{P}'_{l+i}$, while the right-hand side takes $\mathcal{P}_{l}$ to $\mathcal{P}'_{l+j}$.  Therefore $i=j$ by Corollary \ref{cor.twosidediso}.

Next, we must show that $\delta_{1}=\delta_{2}$ and $\epsilon_{1}=\epsilon_{2}$.  To accomplish this, we show that if $F=\Phi T_{\delta,\epsilon}$ where $\Phi$ is induced by an isomorphism of two-sided vector spaces $\phi:V \rightarrow W$, then the sequence $\tau$ associated to $F$ by Theorem \ref{thm.twosidedisom} is
$$
\tau_{i} = \begin{cases} \delta^{-1} \mbox{ if $i$ is even} \\ \epsilon^{-1} \mbox{ if $i$ is odd.} \end{cases}
$$
The result will then follow from Theorem \ref{thm.twosidedisom}.  By Lemma \ref{lemma.functorvalues1}(1) and Lemma \ref{lemma.functorvalues2}(1), the integer $i$ in Theorem \ref{thm.twosidedisom} is $0$.  After some preliminary steps, we will compute $t_{jj}:\mathcal{A}_{jj} \rightarrow \mathcal{A'}_{jj}$ where $t$ is defined in the proof of Theorem \ref{thm.twosidedisom}, and will use this computation to determine the sequence $\tau$.  Throughout the proof we let $T = T_{\delta,\epsilon}$, we let
$$
\zeta_{i} = \begin{cases} \delta \mbox{ if $i$ is even} \\ \epsilon \mbox{ if $i$ is odd,} \end{cases}
$$
we let $U=K_{\delta^{-1}} \otimes V \otimes K_{\epsilon}$, we let $Q_{i} = e_{i}\mathbb{S}(U)$ and we let $\mathcal{Q}_{i}=\pi Q_{i}$.  By Lemma \ref{lemma.functorvalues1}(1) and Lemma \ref{lemma.functorvalues2}(1), there exist isomorphisms  $h_{1}:\Phi Q_{i} \rightarrow P'_{i}$ and $h_{2}:TP_{i} \rightarrow Q_{i}$.
\newline
\newline
\noindent
{\it Step 1:  We note that the diagram
$$
\begin{CD}
\Hom_{\mathbb{P}(U)}(T \mathcal{P}_{i}, T \mathcal{P}_{i}) & \overset{\Phi}{\rightarrow} & \Hom_{\mathbb{P}(W)}(\Phi T \mathcal{P}_{i}, \Phi T \mathcal{P}_{i}) & \rightarrow & \Hom_{\mathbb{P}(W)}(\Phi \mathcal{Q}_{i}, \Phi \mathcal{Q}_{i}) \\
@VVV & & @VVV \\
\Hom_{\mathbb{P}(U)}(\mathcal{Q}_{i}, \mathcal{Q}_{i}) & \underset{\Phi}{\rightarrow} & \Hom_{\mathbb{P}(W)}(\Phi \mathcal{Q}_{i}, \Phi \mathcal{Q}_{i}) & \rightarrow & \Hom_{\mathbb{P}(W)}(\mathcal{P}'_{i}, \mathcal{P}'_{i})
\end{CD}
$$
whose left vertical is induced by $h_{2}$, whose upper right horizontal is induced by $h_{2}$ and whose bottom right horizontal and right vertical are induced by $h_{1}$, commutes.}  The straightforward verification is left to the reader.
\newline
\newline
\noindent
{\it Step 2: We note that the diagram
$$
\begin{CD}
\Hom_{{\sf Gr }\mathbb{S}(U)}(Q_{i},Q_{i}) & \overset{\Phi}{\rightarrow} & \Hom_{{\sf Gr }\mathbb{S}(W)}(\Phi Q_{i}, \Phi Q_{i}) & \rightarrow & \Hom_{{\sf Gr }\mathbb{S}(W)}(P'_{i}, P'_{i}) \\
@V{\pi}VV & & @VV{\pi}V \\
\Hom_{\mathbb{P}(U)}(\mathcal{Q}_{i},\mathcal{Q}_{i}) & \underset{\Phi}{\rightarrow} & \Hom_{\mathbb{P}(W)}(\Phi \mathcal{Q}_{i},\Phi \mathcal{Q}_{i}) & \rightarrow & \Hom_{\mathbb{P}(W)}(\mathcal{P}'_{i},\mathcal{P}'_{i})
\end{CD}
$$
whose right horizontal maps are induced by $h_{2}$, commutes.}  The routine proof, which employs Lemma \ref{lemma.pretwosidediso}, is omitted.
\newline
\newline
\noindent
{\it Step 3: The diagram
$$
\begin{CD}
\Hom_{{\sf Gr }\mathbb{S}(V)}(P_{i},P_{i}) & \overset{T}{\rightarrow} &  \Hom_{{\sf Gr }\mathbb{S}(U)}(TP_{i},TP_{i}) & \rightarrow & \Hom_{{\sf Gr }\mathbb{S}(U)}(Q_{i},Q_{i}) \\
@V{\pi}VV & & @VV{\pi}V \\
\Hom_{\mathbb{P}(V)}(\mathcal{P}_{i}, \mathcal{P}_{i}) & \underset{T}{\rightarrow} &  \Hom_{\mathbb{P}(U)}(T \mathcal{P}_{i},T \mathcal{P}_{i}) & \rightarrow & \Hom_{\mathbb{P}(U)}(\mathcal{Q}_{i},\mathcal{Q}_{i})
\end{CD}
$$
whose unlabeled arrows are induced by $h_{2}$, commutes.}  The straightforward proof is similar to the proof of Step 2, and is omitted.
\newline
\newline
\noindent
{\it Step 4: Let $H$ denote the algebra described in Section \ref{section.isom}.  The composition $t_{ii} \pi: H_{ii} \rightarrow \mathcal{A'}_{ii}$ defined in the proof of Theorem \ref{thm.twosidedisom} is equal to the composition
\begin{eqnarray*}
\operatorname{Hom }_{{\sf Gr }\mathbb{S}(V)}(P_{i}, P_{i}) & \overset{T}{\rightarrow} & \operatorname{Hom }_{{\sf Gr }\mathbb{S}(U)}(T P_{i}, T P_{i}) \\
& \rightarrow & \operatorname{Hom }_{{\sf Gr }\mathbb{S}(U)}(Q_{i}, Q_{i}) \\
& \overset{\Phi}{\rightarrow} & \operatorname{Hom }_{{\sf Gr }\mathbb{S}(W)}(\Phi Q_{i}, \Phi Q_{i}) \\
& \rightarrow & \operatorname{Hom }_{{\sf Gr }\mathbb{S}(W)}(P'_{i}, P'_{i}) \\
& \overset{\pi}{\rightarrow} & \operatorname{Hom }_{\mathbb{P}(W)}(\mathcal{P}'_{i}, \mathcal{P}'_{i})
\end{eqnarray*}
whose second arrow is induced by $h_{2}$ and whose fourth arrow is induced by $h_{1}$.}  This follows from Step1, Step 2 and Step 3.
\newline
\newline
\noindent
{\it Step 5:  We compute $\tau_{i}$.}  Let $c \in K$, and let ${}_{c}\mu \in \operatorname{Hom}_{{\sf Gr }\mathbb{S}(V)}(P_{i}, P_{i})$ denote left multiplication by $c$.  By Step 4, $t_{ii} \pi ({}_{c}\mu) = \pi (h_{1} \Phi(h_{2}) \Phi T ({}_{c}\mu)\Phi( h_{2}^{-1}) h_{1}^{-1})$. Therefore, by the proof of Theorem \ref{thm.twosidedisom},
\begin{equation} \label{eqn.biglittlecomp}
\tau_{i}(c) = h_{1} \Phi(h_{2}) \Phi T ({}_{c}\mu)\Phi( h_{2}^{-1}) h_{1}^{-1}(1).
\end{equation}
We explicitly compute the right-hand side of (\ref{eqn.biglittlecomp}).  By Lemma \ref{lemma.functorvalues1}(2) and Lemma \ref{lemma.functorvalues2}(2), there exist nonzero $a, b \in K$ such that $(h_{1})_{i}$ is left multiplication by $a$, ${}_{a}\mu$, and $(h_{2})_{i}={}_{b}\mu \zeta_{i}^{-1}$ where ${}_{b}\mu$ denotes left multiplication by $b$.  Recall that $\Phi$ doesn't change the underlying set of a module or the value of a morphism. Similarly, if $h:M \rightarrow N$ is a morphism in ${\sf Gr }\mathbb{S}(V)$, then $T(h)_{i}:M \otimes K_{\zeta_{i}} \rightarrow N \otimes K_{\zeta_{i}}$ is the map $h_{i}\otimes K_{\zeta_{i}}$.  It then follows from direct calculation that $\tau_{i}(c)=\zeta_{i}^{-1}(c)$ and the result follows.
\end{proof}

Summarizing our results in this section, we have the following
\begin{thm} \label{thm.finalmain}
Suppose $\operatorname{char }k \neq 2$ and let $F:\mathbb{P}(V) \rightarrow \mathbb{P}(W)$ be a $k$-linear equivalence.  Up to natural equivalence, there exists a unique $i \in \mathbb{Z}$, unique $\delta, \epsilon \in \mbox{Gal }(K/k)$, and an isomorphism $\phi:K_{\delta^{-1}} \otimes V \otimes K_{\epsilon} \rightarrow W^{i^*}$ inducing an equivalence $\Phi: \mathbb{P}(K_{\delta^{-1}} \otimes V \otimes K_{\epsilon}) \rightarrow \mathbb{P}(W^{i*})$ such that
$$
F \cong [-i] \circ \Phi \circ T_{\delta, \epsilon}.
$$
Furthermore, $\Phi$ is determined up to inner automorphism in the sense that if $\Phi'$ is induced by $\phi':  K_{\delta^{-1}} \otimes V \otimes K_{\epsilon} \rightarrow W^{i^*}$, then $\Phi$ is naturally equivalent to $\Phi'$ if and only if there exist nonzero $a, b \in K$ such that $\phi' \phi^{-1}(w) = a \cdot w \cdot b$ for all $w \in W^{i*}$.
\end{thm}

\begin{proof}
The existence of the factorization of $F$ follows from Theorem \ref{thm.finalfactor}, the uniqueness of $i$, $\delta$, and $\epsilon$ follow from Proposition \ref{prop.unique}, and the classification of possible $\Phi$ follows from Lemma \ref{lemma.bigphi}(2) and Lemma \ref{lemma.finalbigphi}.
\end{proof}

\section{Automorphism groups of noncommutative projective lines}
For the readers convenience we recall the following notation from the introduction:  We define the {\it automorphism group of }$\mathbb{P}(V)$, denoted $\operatorname{Aut }\mathbb{P}(V)$, to be the set of $k$-linear shift-free equivalences $\mathbb{P}(V) \rightarrow \mathbb{P}(V)$ modulo natural equivalence, with composition induced by composition of functors.  Let $\operatorname{Stab}V$ denote the subgroup of $\mbox{Gal }(K/k)^{2}$ consisting of $(\delta, \epsilon)$ such that $K_{\delta^{-1}} \otimes V \otimes K_{\epsilon} \cong V$ and let $\operatorname{Aut }V$ denote the set of two-sided vector spaces isomorphisms $V \rightarrow V$ modulo the relation defined by setting $\phi' \equiv \phi$ if and only if there exist nonzero $a, b \in K$ such that $\phi'  \phi^{-1}(v)=a \cdot v \cdot b$ for all $v \in V$.  Our goal in this section, realized in Theorem \ref{thm.finalfinalfinal}, is to compute $\operatorname{Aut }\mathbb{P}(V)$.

In this section we will assume $\operatorname{char }k \neq 2$ so that we may employ a divide-and-conquer strategy using the cases in Lemma \ref{lemma.twosidedclass}.  We advise the reader to recall the definition of $V(\lambda)$ and $K(\lambda)$ introduced after the statement of Theorem \ref{thm.papp} from the introduction.  We also introduce the following notation: if $\delta \in \operatorname{Gal }(K/k)$ and $\gamma \in \operatorname{Emb}(K)$, we let ${}_{\delta}K \vee \gamma(K)_{\gamma}$ denote the two-sided vector space over $K$ with underlying set $K \vee \gamma(K)$ and $K$-action defined by $a \cdot v \cdot b := av\gamma(b)$.  Thus, ${}_{1}K\vee \lambda(K)_{\lambda} = V(\lambda)$.

\begin{prop} \label{prop.canonicalmaps}
If $(\delta, \epsilon) \in \operatorname{Stab}V$ then there exists a canonical isomorphism
$$
\psi_{\delta,\epsilon}:V \rightarrow K_{\delta^{-1}} \otimes V \otimes K_{\epsilon}.
$$
Therefore, there is a canonical bijection $\operatorname{Aut }\mathbb{P}(V) \rightarrow \operatorname{Aut}V \times \operatorname{Stab}V$.
\end{prop}

\begin{proof}
We prove the first part in three cases corresponding to the structure of $V$ according to Lemma \ref{lemma.twosidedclass}.

First, suppose there exists a $\sigma \in \operatorname{Gal}(K/k)$ such that $V=K_{\sigma} \oplus K_{\sigma}$.  The fact that $(\delta,\epsilon) \in \operatorname{Stab}V$ then implies that $\epsilon=\sigma^{-1} \delta \sigma$, and it is thus straightforward to check that the map
$$
\psi_{\delta, \epsilon}:V \rightarrow K_{\delta^{-1}} \otimes V \otimes K_{\epsilon}
$$
defined by $(a,b) \mapsto 1 \otimes (\delta(a),\delta(b)) \otimes 1$ is an isomorphism of two-sided vector spaces.

Next, suppose there exist $\sigma, \tau \in \operatorname{Gal}(K/k)$ such that $\sigma \neq \tau$ and $V=K_{\sigma} \oplus K_{\tau}$, and suppose $(\delta, \epsilon) \in \operatorname{Stab}V$.  There are two possibilities.  If $\delta^{-1} \sigma \epsilon= \sigma$ and $\delta^{-1} \tau \epsilon = \tau$, then $\epsilon=\sigma^{-1} \delta \sigma=\tau^{-1} \delta \tau$, and the map
$$
\psi_{\delta, \epsilon}:V \rightarrow K_{\delta^{-1}} \otimes V \otimes K_{\epsilon}
$$
defined by $(a,b) \mapsto 1 \otimes (\delta(a),\delta(b)) \otimes 1$ is an isomorphism of two-sided vector spaces.

If $\delta^{-1} \sigma \epsilon =\tau$ and $\delta^{-1} \tau \epsilon = \sigma$, then $\epsilon=\tau^{-1} \delta \sigma=\sigma^{-1} \delta \tau$, and the map
$$
\psi_{\delta, \epsilon}:V \rightarrow K_{\delta^{-1}} \otimes V \otimes K_{\epsilon}
$$
defined by $(a,b) \mapsto 1 \otimes (\delta(b),\delta(a)) \otimes 1$ is an isomorphism of two-sided vector spaces.

Finally, suppose $V={}_{1}K \vee \lambda(K)_{\lambda}$.  Suppose, further, that $(\delta, \epsilon) \in \operatorname{Stab}V$.  It follows from Lemma \ref{lemma.autocomp} that ${\gamma}^{-1} \lambda \epsilon \in \lambda^{G}$ for some extension $\gamma$ of $\delta$.  Therefore, there exists an extension $\overline{\delta}$ of $\delta$ such that ${\overline{\delta}}^{-1} \lambda \epsilon = \lambda$.  We claim that $\overline{\delta}|_{K(\lambda)}$ is a $k$-linear automorphism of $K(\lambda)$ and that if $\overline{\delta}'$ is another extension of $\delta$ with the property that $(\overline{\delta}')^{-1}\lambda \epsilon = \lambda$ then $\overline{\delta}'|_{K(\lambda)}=\overline{\delta}|_{K(\lambda)}$.  To prove the claims, let $a \in K$ be such that $K=k[a]$.  Then $\lambda(a)$ is not in $K$, and every element of $K(\lambda)$ can be written uniquely in the form $f(a)+g(a)\lambda(a)$ for some $f,g \in k[x]$.  We have
\begin{eqnarray*}
\overline{\delta}(f(a)+g(a)\lambda(a)) & = & f(\delta(a))+g(\delta(a))\overline{\delta}(\lambda(a)) \\
& = & f(\delta(a))+g(\delta(a))\lambda(\epsilon(a)) \in K(\lambda).
\end{eqnarray*}
The first claim follows immediately, while the second claim follows from the fact that if the above computation is preformed using $\overline{\delta}'$ in place of $\overline{\delta}$, the same outcome occurs.

We define $\psi:K(\lambda) \rightarrow K(\lambda)$ to be the function $\overline{\delta}|_{K(\lambda)}$ where $\overline{\delta}$ is some extension of $\delta$ such that ${\overline{\delta}}^{-1} \lambda \epsilon = \lambda$.

We next claim that $\psi$ induces an isomorphism of two-sided vector spaces
$$
\psi:{}_{1}K \vee \lambda(K)_{\lambda} \rightarrow {}_{\delta}K \vee \lambda\epsilon(K)_{\lambda \epsilon}.
$$
The proof is routine and omitted.

Finally, we claim that ${}_{\delta}K \vee \lambda\epsilon(K)_{\lambda \epsilon}$ is canonically isomorphic to $K_{\delta^{-1}} \otimes V \otimes K_{\epsilon}$.  To prove the claim, one checks that the function
\begin{equation} \label{eqn.precomp}
{}_{\delta}K \vee \lambda \epsilon (K)_{\lambda \epsilon} \rightarrow K_{\delta^{-1}} \otimes V \otimes K_{\epsilon}
\end{equation}
defined by sending $v$ to $1 \otimes v \otimes 1$ is a two-sided vector space isomorphism.  We define $\psi_{\delta, \epsilon}$ to be the composition of (\ref{eqn.precomp}) with $\psi$.

For the second part of the proposition, we note that if $F:\mathbb{P}(V) \rightarrow \mathbb{P}(V)$ is a shift-free $k$-linear equivalence, then by Theorem \ref{thm.finalmain}, $F \cong \Phi T_{\delta, \epsilon}$, where $\Phi:\mathbb{P}(K_{\delta^{-1}} \otimes V \otimes K_{\epsilon}) \rightarrow \mathbb{P}(V)$ is induced by an isomorphism of two-sided vector spaces $\phi:K_{\delta^{-1}} \otimes V \otimes K_{\epsilon} \rightarrow V$.  We define a function
$$
\operatorname{Aut }\mathbb{P}(V) \rightarrow \operatorname{Aut}V \times \operatorname{Stab}V
$$
by sending the class of $\Phi T_{\delta, \epsilon}$ to $([\phi  \psi_{\delta, \epsilon}], (\delta, \epsilon))$.
\end{proof}
We now turn to the computation of the groups $\operatorname{Aut }V$ and $\operatorname{Stab }V$.
\begin{lemma} \label{lemma.aut}
The group $\operatorname{Aut}(V)$ is the following:
\begin{enumerate}
\item{} if there exists $\sigma \in \operatorname{Gal}(K/k)$ such that $V=K_{\sigma} \oplus K_{\sigma}$ then
    $$
    \operatorname{Aut}(V) \cong \mbox{PGL}_{2}(K),
    $$

\item{} if there exists $\sigma, \tau \in \operatorname{Gal}(K/k)$, $\sigma \neq \tau$ and $V=K_{\sigma} \oplus K_{\tau}$ then
    $$
    \operatorname{Aut}(V) \cong K^{*} \times K^{*}/\{(a \sigma(b),a \tau(b))|a,b \in K^{*}\}
    $$

\item{} if $V={}_{1}K \vee \lambda(K)_{\lambda}$ then
$$
\operatorname{Aut}(V) \cong K(\lambda)^{*}/K^{*}\lambda(K)^{*}.
$$
\end{enumerate}
\end{lemma}

\begin{proof}
The result in (1) is elementary.  We now prove the result in case there exists $\sigma, \tau \in \operatorname{Gal}(K/k)$, $\sigma \neq \tau$ and $V=K_{\sigma} \oplus K_{\tau}$.  In this case, the function $K \times K \rightarrow \operatorname{Hom }_{K\otimes_{k}K}(K_{\sigma} \oplus K_{\tau}, K_{\sigma} \oplus K_{\tau})$ sending $(a,b)$ to ${}_{a}\mu \oplus {}_{b}\mu$, where ${}_{a}\mu:K_{\sigma} \rightarrow K_{\sigma}$ and ${}_{b}\mu:K_{\tau} \rightarrow K_{\tau}$ denote left multiplication by $a$ and $b$, is a bijection compatible with multiplication, and the result follows easily from this.

Finally, in case $V={}_{1}K \vee \lambda(K)_{\lambda}$, it follows from \cite[Proposition 3.6]{papp} that the function $K(\lambda) \rightarrow \operatorname{Hom}_{K\otimes_{k}K}(V,V)$ defined by sending $a$ to ${}_{a}\mu$ is a bijection compatible with multiplication, and the result follows.
\end{proof}

\begin{lemma} \label{lemma.stab}
The group $\operatorname{Stab }(V)$ is the following:
\begin{enumerate}
\item{} if there exists $\sigma \in \operatorname{Gal}(K/k)$ such that $V=K_{\sigma} \oplus K_{\sigma}$ then
    $$
    \operatorname{Stab}(V) \cong \operatorname{Gal }(K/k),
    $$

\item{} if there exists $\sigma, \tau \in \operatorname{Gal}(K/k)$, $\sigma \neq \tau$, and $V=K_{\sigma} \oplus K_{\tau}$, then
    $$
    \operatorname{Stab}(V) = \{(\delta , \epsilon) \in \operatorname{Gal}(K/k)^{2} | \{\delta^{-1} \sigma \epsilon, \delta^{-1} \tau \epsilon \}=\{\sigma, \tau\} \}
    $$

\item{} if $V={}_{1}K \vee \lambda(K)_{\lambda}$ then
$$
\operatorname{Stab}(V) = \{(\delta, \epsilon) \in \operatorname{Gal}(K/k)^{2} | (\overline{\delta}^{-1}\lambda \epsilon)^{G}=\lambda^{G}\}.
$$
\end{enumerate}

\end{lemma}

\begin{proof}
The result in cases (1) and (2) follows from the fact that if $\sigma, \sigma' \in \operatorname{Gal }(K/k)$, then $K_{\sigma} \cong K_{\sigma'}$ if and only if $\sigma=\sigma'$.  In case (3), the result is an application of Lemma \ref{lemma.autocomp}.
\end{proof}
In preparation for Theorem \ref{thm.finalfinalfinal}, we define a group homomorphism
$$
\theta: \operatorname{Stab}(V)^{op} \rightarrow \operatorname{Aut}(\operatorname{Aut}(V))
$$
as follows:

\begin{enumerate}
\item{} In case $\sigma \in \operatorname{Gal}(K/k)$, $V=K_{\sigma} \oplus K_{\sigma}$, and $(\delta, \epsilon) \in \operatorname{Stab}V$, we define $\theta(\delta, \epsilon)$ by sending the class of a matrix $(a_{ij})$ to the class of $(\delta^{-1}(a_{ij}))$, i.e. $\theta(\delta, \epsilon)$ acting on the class of an isomorphism $\phi$ equals the class of $\delta^{-1} \phi \delta$, where $\delta^{-1}$ and $\delta$ act coordinate-wise on $K^{2}$.

\item{} In case $\sigma, \tau \in \operatorname{Gal}(K/k)$, $\sigma \neq \tau$, $V=K_{\sigma} \oplus K_{\tau}$, and $(\delta, \epsilon) \in \operatorname{Stab}V$ is such that $\delta^{-1} \sigma \epsilon = \sigma$, we define $\theta(\delta, \epsilon)$ as the element in
    $$
    \operatorname{Aut}(K^{*} \times K^{*}/\{(a \sigma(b),a \tau(b))|a,b \in K^{*}\})
    $$
    which sends the class of the pair $(c, d) \in K^{*} \times K^{*}$ to the class of the pair $(\delta^{-1}(c), \delta^{-1}(d))$.  In other words, $\theta(\delta, \epsilon)$ acting on the class of an isomorphism $\phi$ equals the class of $\delta^{-1} \phi \delta$, where $\delta^{-1}$ and $\delta$ act coordinate-wise on $K^{2}$.

    If $\delta^{-1} \sigma \epsilon = \tau$, then we define $\theta(\delta, \epsilon)$ as automorphism sending the class of $(c, d)$ to the class of $(\delta^{-1}(d), \delta^{-1}(c))$.  In other words, $\theta(\delta, \epsilon)$ acting on the class of an isomorphism $\phi$ equals the class of $\upsilon \delta^{-1} \phi \delta \upsilon$, where $\delta^{-1}$ and $\delta$ act coordinate-wise on $K^{2}$ and $\upsilon$ is the linear transformation that exchanges the coordinates.

\item{} In case $V={}_{1}K \vee \lambda(K)_{\lambda}$ and $(\delta, \epsilon) \in \operatorname{Stab}V$, we define $\theta(\delta, \epsilon)$ as the function sending the class of $c \in K(\lambda)^{*}$ to the class of $\psi^{-1}(c)$, where $\psi:K(\lambda) \rightarrow K(\lambda)$ is the $k$-algebra isomorphism defined in Proposition \ref{prop.canonicalmaps}.  Thus, $\theta(\delta, \epsilon)$ acting on the class of an isomorphism $\phi$ equals the class of $\psi^{-1} \phi \psi$.
\end{enumerate}
It is straightforward to check that the function $\theta$ defined above is a group homomorphism.

\begin{thm} \label{thm.finalfinalfinal}
The bijection from Proposition \ref{prop.canonicalmaps} induces an isomorphism of groups $\operatorname{Aut }\mathbb{P}(V) \rightarrow \operatorname{Aut }V \rtimes_{\theta} \operatorname{Stab}V^{op}$. \end{thm}

\begin{proof}
We describe the product structure on $\operatorname{Aut }\mathbb{P}(V)$.  For $i=1,2$, we have functors
$$
\mathbb{P}(V) \overset{T_{\delta_{i},\epsilon_{i}}}{\rightarrow} \mathbb{P}(K_{\delta_{i}^{-1}} \otimes V \otimes K_{\epsilon_{i}}) \overset{\Phi_{i}}{\rightarrow} \mathbb{P}(V)
$$
where $\Phi_{i}$ is induced by an isomorphism of two-sided vector spaces $K_{\delta_{i}^{-1}} \otimes V \otimes K_{\epsilon_{i}} \rightarrow V$.  We let
$$
S_{\delta_{2},\epsilon_{2}}:\mathbb{P}(K_{\delta_{1}^{-1}} \otimes V \otimes K_{\epsilon_{1}}) \rightarrow \mathbb{P}(K_{\delta_{2}^{-1}}\otimes K_{\delta_{1}^{-1}} \otimes V \otimes K_{\epsilon_{1}} \otimes K_{\epsilon_{2}})
$$
denote the twist functor, we let $S'$ be a left and right adjoint of $S_{\delta_{2},\epsilon_{2}}$ and we let $\tilde{\Phi}_{1}:\mathbb{P}(K_{\delta_{2}^{-1}} \otimes K_{\delta_{1}^{-1}} \otimes V \otimes K_{\epsilon_{1}} \otimes K_{\epsilon_{2}}) \rightarrow \mathbb{P}(K_{\delta_{2}^{-1}} \otimes V \otimes K_{\epsilon_{2}})$ denote the equivalence induced by $K_{\delta_{2}^{-1}} \otimes \phi_{1} \otimes K_{\epsilon_{2}}$.  Then we have
\begin{eqnarray*}
\Phi_{2} T_{\delta_{2},\epsilon_{2}} \Phi_{1} T_{\delta_{1},\epsilon_{1}} & \cong & \Phi_{2} T_{\delta_{2},\epsilon_{2}} \Phi_{1} S'S_{\delta_{2},\epsilon_{2}}T_{\delta_{1},\epsilon_{1}} \\
& \cong & \Phi_{2} \tilde{\Phi}_{1} S_{\delta_{2},\epsilon_{2}}T_{\delta_{1},\epsilon_{1}} \\
& \cong & \Phi_{2} \tilde{\Phi}_{1} \Phi' T_{\delta_{1}\delta_{2}, \epsilon_{1} \epsilon_{2}}
\end{eqnarray*}
where the second natural equivalence is from the first part of Lemma \ref{lemma.semidirect}, $\Phi'$ is induced by the canonical isomorphism of two-sided vector spaces
$$
K_{(\delta_{1} \delta_{2})^{-1}} \otimes V \otimes K_{\epsilon_{1} \epsilon_{2}} \rightarrow K_{\delta_{2}^{-1}}\otimes K_{\delta_{1}^{-1}} \otimes V \otimes K_{\epsilon_{1}} \otimes K_{\epsilon_{2}},
$$
and the third natural equivalence is from the second part of Lemma \ref{lemma.semidirect}.  Now, $\Phi_{2} \tilde{\Phi}_{1} \Phi'$ is naturally equivalent to the functor induced by the isomorphism of two-sided vector spaces
$$
K_{(\delta_{1} \delta_{2})^{-1}} \otimes V \otimes K_{\epsilon_{1} \epsilon_{2}} \rightarrow K_{\delta_{2}^{-1}}\otimes K_{\delta_{1}^{-1}} \otimes V \otimes K_{\epsilon_{1}} \otimes K_{\epsilon_{2}} \overset{K_{\delta_{2}^{-1}} \otimes \phi_{1} \otimes K_{\epsilon_{2}}}{\rightarrow} K_{\delta_{2}^{-1}} \otimes V \otimes K_{\epsilon_{2}} \overset{\phi_{2}}{\rightarrow} V
$$
whose left arrow is canonical.  If we call the above composition $\phi$, then under the bijection from Proposition \ref{prop.canonicalmaps},
$$
[\Phi_{2}T_{\delta_{2},\epsilon_{2}}][\Phi_{1}T_{\delta_{1}, \epsilon_{1}}] \mapsto ([\phi \psi_{\delta_{1}\delta_{2}, \epsilon_{1}\epsilon_{2}}], (\delta_{1}\delta_{2}, \epsilon_{1} \epsilon_{2})) \in \operatorname{Aut }V \times \operatorname{Stab }V.
$$
Therefore, to complete the proof of the theorem, it remains to check that
$$
[\phi_{2} \psi_{\delta_{2},\epsilon_{2}}] \theta_{\delta_{2},\epsilon_{2}}([\phi_{1}\psi_{\delta_{1},\epsilon_{1}}])=[\phi \psi_{\delta_{1}\delta_{2},\epsilon_{1} \epsilon_{2}}]
$$
in $\operatorname{Aut }(V)$.  This follows from an explicit case-by-case computation.  The straightforward details are left to the reader.
\end{proof}

\end{document}